%% file: DLSRMStationaryn-20260914V54-b.tex
\documentclass[a4paper,12pt]{article}
\usepackage{dsfont}
\usepackage{amssymb}
\usepackage{latexsym}
\usepackage{amsmath}
\usepackage{color}
\usepackage{comment}
\usepackage{amsthm}
\usepackage[dvips]{graphicx}
\usepackage{layout} 
\usepackage{ulem}
\usepackage{enumerate}
\usepackage{bm}
\usepackage{bbm} 
\usepackage[bbgreekl]{mathbbol} 
\usepackage{authblk}
\usepackage{setspace} 
\newif\ifrs
\rstrue
\ifrs \usepackage{mathrsfs} \fi  
\newif\ifcol
\coltrue 

\newtheorem{theorem}{Theorem}[section]

\newtheorem{lemma}[theorem]{Lemma}

\newtheorem{proposition}[theorem]{Proposition}

\newtheorem{remark}[theorem]{Remark}
\newtheorem{example}[theorem]{Example}
\numberwithin{equation}{section}
\newtheorem{theorem*}{Theorem}
\newtheorem{ass*}[theorem*]{Assumption}
\newtheorem{note*}[theorem*]{Note}
\newtheorem{lemma*}[theorem*]{Lemma}
\newtheorem{definition*}[theorem*]{Definition}
\newtheorem{proposition*}[theorem*]{Proposition}
\newtheorem{corollary*}[theorem*]{Corollary}
\newtheorem{remark*}[theorem*]{Remark}
\newtheorem{example*}[theorem*]{Example}
\numberwithin{equation}{section}
\newcommand{\varep}{\varepsilon}

\newcommand{\constad}{c_x}

\def\tm{{t_-}}
\def\DL{0}
\def\sfB{{\bbB}}
\newcommand{\tsubo}{0}
\newcommand\ttU{{\tt U}}
\newcommand\bbd{{\mathbbm d}}

\newcommand\Csa{C_1}
\newcommand\Csb{C_2}
\newcommand\Csc{C_3}
\newcommand\Csd{C_4}
\newcommand\Cse{C_5}
\newcommand\Csf{C_6}
\newcommand\Csg{C_7}
\newcommand\Csh{C_8}
\newcommand\Csi{C_{58}}
\newcommand\csj{c_{59}}
\newcommand\Csk{C_{18}}
\newcommand\Csl{C_{11}}
\newcommand\Csm{C_{12}}
\newcommand\Csn{C_{13}}
\newcommand\Cso{C_{14}}
\newcommand\Csp{C_{9}}
\newcommand\Csq{C_{15}}
\newcommand\Csr{C_{16}}
\newcommand\Css{c_{22}}
\newcommand\Cst{c_{18}}
\newcommand\Csu{C_{5}}
\newcommand\Csv{C_{20}}
\newcommand\Csw{C_{21}}
\newcommand\Csy{C_{23}}
\newcommand\Csz{C_{24}}
\newcommand\Csaa{C_{25}}
\newcommand\Csab{C_{4}}
\newcommand\Csac{C_{6}}
\newcommand\Csad{C_{8}}
\newcommand\Csae{C_{12}}
\newcommand\Csaf{c_{11}}
\newcommand\Csag{C_{13}}
\newcommand\Csah{C_{7}}
\newcommand\Csai{C_{9}}
\newcommand\Csaj{C_{3}}
\newcommand\Csak{C_{1}}
\newcommand\Csal{C_{14}}
\newcommand\Csam{C_{37}}
\newcommand\Csan{C_{38}}
\newcommand\Csao{C_{10}}
\newcommand\Csap{C_{40}}
\newcommand\Csaq{C_{41}}
\newcommand\Csar{C_{16}}
\newcommand\Csas{c_{23}}
\newcommand\Csat{C_{25}}
\newcommand\Csau{C_{45}}
\newcommand\Csaw{C_{47}}
\newcommand\Csax{C_{15}}
\newcommand\Csay{C_{19}}
\newcommand\Csaz{C_{20}}
\newcommand\Csba{C_{27}}
\newcommand\Csbb{C_{26}}
\newcommand\Csbc{C_{2}}
\newcommand\Csbd{C_{54}}
\newcommand\Csbe{C_{55}}
\newcommand\Csbf{C_{56}}
\newcommand\Csbg{C_{57}}
\newcommand\Csbh{C_{17}}
\newcommand\Csbi{C_{21}}
\newcommand\Csbj{C_{25}}
\newcommand\Csbk{c_{24}}
\newcommand\Csbl{C_{28}}

\newcommand\constzz{{C_{*}}}

\newcommand\sfq{{\sf q}}
\newcommand\bfR{{\bf R}}
\newcommand\ttT{{\tt T}}

\newcommand\onecalw{\calw}
\newif\ifcol
\colfalse
\ifcol
\newcommand{\colorr}{\color{black}}
\newcommand{\colorg}{\color[rgb]{0,0.5,0}}
\newcommand{\colorb}{\color[rgb]{0,0,0.8}}

\newcommand{\colorn}{\color[rgb]{1,1,1}}

\newcommand{\colory}{\color{yellow}}
\newcommand{\coloroy}{\color[rgb]{1,0.95,0}}

\newcommand{\colorro}{\color[rgb]{0.851,0.255,0.467}} 

\else
\newcommand{\colorb}{\color{black}}

\newcommand{\colorr}{\color{black}}
\newcommand{\colorg}{\color{black}}

\newcommand{\colorn}{\color{black}}
\newcommand{\colory}{\color{black}}

\newcommand{\coloroy}{\color{black}}
\newcommand{\colorro}{\color{black}}
\fi
\newif\ifcol
\colfalse
\ifcol
\newcommand{\cred}{\color[rgb]{0.8,0,0}}

\newcommand{\cblue}{\color[rgb]{0,0,0.8}}
\else
\newcommand{\cred}{\color{black}}
\newcommand{\cblue}{\color{black}}
\fi
\newif\ifcol
\colfalse
\ifcol

\newcommand{\tred}{\color[rgb]{0.8,0,0}}

\else
\newcommand{\tred}{\color{black}}
\fi
\newif\ifcol
\colfalse
\ifcol

\newcommand{\fred}{\color[rgb]{0.8,0,0}}

\else
\newcommand{\fred}{\color{black}}
\fi
\newif\ifcol
\coltrue
\ifcol

\else
\fi
\newif\ifcol
\colfalse
\ifcol

\newcommand{\ared}{\color[rgb]{0.8,0,0}}

\newcommand{\ablue}{\color[rgb]{0,0,0.8}}

\else
\newcommand{\ared}{\color{black}}
\newcommand{\ablue}{\color{black}}
\fi
\newif\ifcol
\colfalse
\ifcol

\newcommand{\bred}{\color[rgb]{0.8,0,0}}

\else
\newcommand{\bred}{\color{black}}
\fi
\excludecomment{en-text}
\includecomment{jp-text}
\includecomment{comment}
\input nakamacro300823-300916+.tex

\newcommand{\ol}{\overline}

\renewcommand{\colory}{\color{yellow}}
\renewcommand{\koko}{{\colory koko}}
\newcommand{\kokokara}{{\colory kokokara}}
\newcommand{\wh}{\widehat}
\newcommand{\wt}{\widetilde}

\newcommand{\gray}{\color[rgb]{0.5,0.5,0.5}}

\excludecomment{en-text}
\includecomment{jp-text}
\includecomment{comment}

\newcounter{constant}

\newcommand{\cst}{\stepcounter{constant} \ensuremath{{\tt C}_{\theconstant}}}

\begin{document}

\title{
{\cred Risk comparison theorems and application to deep learning of diffusion coefficients}
\footnote{
This work was in part supported by 
Japan Science and Technology Agency CREST JPMJCR2115;  
Japan Society for the Promotion of Science Grants-in-Aid for Scientific Research 
No. 23H03354 (Scientific Research);  
and by a Cooperative Research Program of the Institute of Statistical Mathematics. 
Sincere gratitude goes to Mr. Yoshito Date (EMC Healthcare Co., Ltd.), Mr. Naokatsu Hasegawa (Yayoi Kogyo Co., Ltd.), and Mr. Masaki Nonaka (TAUNS Laboratories, Inc.) for their valuable support.
}
}
\author[1]{Arnaud Gloter}
\author[2,3]{Nakahiro Yoshida}
\affil[1]{Laboratoire de Math\'ematiques et Mod\'elisation d'Evry, Universit\'e \'Evry Paris-Saclay
	\footnote{
		Laboratoire de Math\'ematiques et Mod\'elisation d'Evry, CNRS, Univ Evry, 
		Universit\'e Evry Paris-Saclay, 91037, Evry, France. e-mail: arnaud.gloter@univ-evry.fr}
}
\affil[2]{Graduate School of Mathematical Sciences, University of Tokyo
\footnote{Graduate School of Mathematical Sciences, University of Tokyo: 3-8-1 Komaba, Meguro-ku, Tokyo 153-8914, Japan. e-mail: nakahiro@ms.u-tokyo.ac.jp}
        }
\affil[3]{The Institute of Statistical Mathematics
        }
\maketitle
\ \\
{\it Summary} \ 
{\cred
We investigate the nonparametric estimation of the diffusion matrix in stochastic differential equations featuring multidimensional, strong mixing covariate processes. We propose a flexible statistical framework based on general function classes that does not require a linear basis representation, rendering our results directly applicable to deep neural network estimators. Our approach employs a two-step estimation procedure: constructing a preliminary 
nonparametric 
quasi-likelihood estimator and subsequently regularizing it via a $\beta$-H\"older class approximation. We establish general risk comparison theorems between empirical and generalization risks for arbitrary estimators without relying on a specific probabilistic structure of the underlying process. 
{\fred In diffusion matrix learning based on \(n + 1\) observations over the time interval \([0,T]\), the derived upper bounds capture the intrinsic interplay between the \(T\)-rate, associated with the mixing behavior of the covariate process, and the intrinsic \(n\)-rate governing the volatility estimation.
}%
%
}

\section{Introduction}

This work concerns the nonparametric estimation of the diffusion coefficient  in stochastic differential equations. We consider the model 
\begin{equation*}
	dY_t=b(X_t) dt + \sigma(X_t) dw_t,
\end{equation*}
where $Y{\cred\>=(Y_t)_{t\in\bbR_+}}$ and $X{\cred\>=(X_t)_{t\in\bbR_+}}$ are multidimensional processes and $w{\cred\>=(w_t)_{t\in\bbR_+}}$ is a {\cred multidimensional} Brownian motion. 
The process $X$ is assumed to {\cred satisfy a strong mixing condition}  
and represents a collection of covariates affecting the volatility of the process $Y$. These covariates may be endogenous or may include components of the process $Y$.  In particular, this framework encompasses the classical diffusion setting for $X=Y$. We assume that high-frequency observations $(X_\tj,Y_\tj)_{j=0,...,n}$
are available, with $h_n\to0$ and ${\cred T\equiv\>}T_n=nh_n\to\infty$. Our objective is to estimate the diffusion matrix $S(x)=\sigma\sigma^\star(x)$ on a compact subset of the state space of $X$.

In the parametric situation, the estimation of the diffusion coefficients from high-frequency observations has been extensively studied; see, for example, 
{\cred 
\cite{genon-catalotEstimationDiffusionCoefficient1993,ogihara2014quasi,uchida2013quasi} for estimation over a finite time horizon, 
 \cite{kesslerEstimationErgodicDiffusion1997,
uchida2012adaptive,yoshidaEstimationDiffusionProcesses1992b} for ergodic diffusion processes, 
\cite{uchida2014adaptive,yoshida2011polynomial} for Bayesian analysis, 
\cite{inatsugu2021global,ogihara2011quasi,shimizu2006estimation} for jump diffusion processes, 
and 
 \cite{gloter2021adaptive,gloter2024non} for degenerate diffusion processes. 
 }%

%
{\cred 
The nonparametric estimation of $S$ is addressed, 
when the process is observed over a fixed interval of time $[0,1]$, in the references \cite{bandiFullyNonparametricEstimation2003, florensEstimationDiffusionCoefficient1998, hoffmannAdaptiveEstimationDiffusion1999a, hoffmannEstimationDiffusion1999, jacodNonparametricKernelEstimation2000}. Since the diffusion coefficient can only be estimated 
in regions visited by the process $X$, these approaches heavily rely on the local time of the process. Consequently, they are inherently confined to one-dimensional diffusions. In this setting, the minimax rate of estimation is $n^{\beta/(1+2\beta)}$, where $\beta$ denotes the smoothness of $S$ (see \cite{hoffmannAdaptiveEstimationDiffusion1999a}).
}

The estimation of the diffusion coefficient for ergodic diffusion processes observed over a long time horizon was studied by \cite{comtePenalizedNonparametricMean2007c}. This framework has subsequently been extended to noisy observations in \cite{schmisserNonparametricEstimationDiffusion2012a}, and to diffusions with jumps in \cite{schmisserNonParametricEstimation2019b} and \cite{amorinoNonparametricInferenceCoefficients2022}.

Recently, statistical inference based on repeated independent diffusion trajectories has attracted growing interest (see, e.g., \cite{comteNonparametricDriftEstimation2020b, genon-catalotEstimationStochasticDifferential2016}). In this setting, the nonparametric estimation of the diffusion coefficient has been investigated by \cite{denisNonparametricPluginClassifier2024} and \cite{ella-mintsaNonparametricEstimationDiffusion2024a}. %
{\cred
It should be stressed that, as far as volatility estimation is concerned, the underlying process is confined to the one-dimensional case in all of these references.}

%
{\cred 
In the present work, we aim to accommodate a more general structure for the process $Z=(X,Y)$ and, in particular, consider multidimensional covariate processes. Moreover, our statistical analysis is formulated for general classes of approximation functions, rendering it readily applicable to deep neural network estimators of the diffusion coefficient.
}

%
{\cred 
It is known that in high-dimensional contexts, a representation by neural networks tends to mitigate the curse of dimensionality in regression problems. It was shown in \cite{schmidt2020nonparametric} that a ReLU (Rectified Linear Unit) deep neural network can achieve a faster rate of convergence than wavelet estimators under structural constraints on the regression function. Approximation properties of neural networks and their statistical implications have been extensively studied in recent years; see, among others,   
\cite{bauerDeepLearningRemedy2019, 
kohlerRateConvergenceFully2021, okoDiffusionModelsAre2023, suzukiAdaptivityDeepReLU2018a, suzukiDeepLearningAdaptive2021}.
A survey of recent developments in deep learning can be found in \cite{devoreNeuralNetworkApproximation2021, fanSelectiveOverviewDeep2021, suhSurveyStatisticalTheory2025}.
}

{\cred Theoretical guarantees for the generalization error of deep neural networks in the context of estimation of stochastic processes remain relatively limited. Some results for times series are established in 
\cite{kengneExcessRiskBound2025, kohlerRateConvergenceDeep2023, kurisuAdaptiveDeepLearning2025}
(see also reference therein).  
Nonparametric estimation of the drift function using deep neural networks has been considered for ergodic diffusion processes in \cite{digregorioNeuralDriftEstimation2025, ogaDriftEstimationMultidimensional2024}, while that for repeated diffusion processes has been investigated in \cite{zhaoDriftEstimationDiffusion2026}. 
Additionally, deep learning approaches for point processes have been explored in \cite{gyotoku2025deep}.}

{\cred
The ultimate goal of this study is to establish such a theoretical framework for estimating the diffusion matrix. Let $\mathfrak{F}_n$ denote a class of candidate functions for the estimation of $S$. This class may, in particular, consist of the outputs generated by a neural network architecture. For any $S \in \mathfrak{F}_n$, we quantify its discrepancy from the true diffusion matrix $S^*$ via a function $U(x,S)$ defined in (\ref{202608011524}). This function satisfies $U(x,S) \asymp \big|S(x)-S^*(x) \big|^2 \onecalw(x)$, where $\onecalw$ is a {\cred non-negative} smooth function with compact support. The approximation error associated with the model class $\mathfrak{F}_n$ is given by $\inf_{S \in \mathfrak{F}_n} n^{-1} E \big[ {\cred \sum_{j=1}^{n} U(X_{t_{j-1}}, S)} \big]$. When the process $X$ is stationary, this quantity is comparable to the $L^2$ error of approximating $S^*$ by the chosen model class $\mathfrak{F}_n$.
}

%
{\cred 
Our estimation procedure consists of two steps. First, we construct a preliminary estimator $\wh{S}_n^\DL$ by minimizing the contrast function $\Phi_n(S,Z)$ defined in \eqref{202505042250}, which is associated with the Gaussian quasi-likelihood of the statistical model. Since the regularity of $\wh{S}_n^\DL$ cannot be guaranteed a priori, we subsequently approximate it by $\widehat{S}_n$, which belongs to a {\cred $\beta$-H\"older class} for some $\beta>0$. 
The corresponding generalization risk is defined as $\ol{\bfR}_n=
E \big[ \Phi_n(\wh{S}_n,\ol{Z})  \big]  - E \big[ \Phi_n(S^*,\ol{Z})  \big] $, 
{\cred where $\ol{Z}=(\ol{X},\ol{Y})$ is an independent copy of the process $Z=(X,Y)$.}
}
{\cred  
For the estimator $\wh{S}_n$, we can also evaluate its distance to $S^*$ through the predictable version of the generalization error, defined as
$\ol{\bfR}_n^p=
n^{-1}E\big[{\cred\sum_{j=1}^{n} U(\ol{X}_\tjm,\wh{S}_n)}\big]$. 
More explicit expressions of $\ol{\bfR}_n$ and $\ol{\bfR}_n^p$ are given in 
(\ref{202603120055}) and (\ref{202608011537}), respectively. 
}

%
{\cred
In Theorem \ref{202602140323}, we establish that the same upper bound holds for the two errors $\overline{\mathbf{R}}_n$ and $\overline{\mathbf{R}}_n^p$. This upper bound reads as
\bea\label{20608011544}
	\Delta_n
	{\cred+\wh{\Delta}_n}
	+\inf_{S\in\mathfrak{F}_n}
	{\cred E\bigg[\frac{1}{n}\sum_{j\in\bbI_n}U(S,X_\tjm)\bigg]}
	+n^{-1}(\log n)\log\caln_n+\delta_n\log n+{\colorro h_n^2},
\eea
where $\caln_n$ denotes the covering number of the class $\mathfrak{F}_n$ by balls of radius $\delta_n$ with respect to the supremum norm, and {\cred$\Delta_n$ and $\wh{\Delta}_n$ are} optimization tolerances, which can be theoretically taken to be $0$ {\cred under perfect optimization}. The validity of this upper bound requires a relationship among $T$, $n$ and $\log\caln_n$, as stated in \eqref{202604131857}. This condition guarantees that $T$ tends to infinity at some minimal rate and is used to compare the generalization errors of the estimator with their empirical counterparts. If this condition is not satisfied, we establish in Proposition \ref{202607131357}
 that an upper bound still holds, but with an additional term $T^{-2\beta/\sfd}$, where $\beta$ represents the smoothness of $S^*$ and $\sfd$ denotes the dimension of $X$. This extra term may dominate when $h_n$ is extremely small, 
and Condition \eqref{202604131857} precisely excludes this regime. 
}

%
{\cred
Such a discrepancy between the empirical and generalization errors does not arise in the neural network estimation of the drift function by \cite{ogaDriftEstimationMultidimensional2024}. This difference appears to result from the interaction between the {\cred${\fred T}$}-rate, which is associated with the mixing behavior of the covariate process $X$, and the intrinsic ${\fred n}$-rate governing the volatility estimation. It should be stressed that the results of \cite{comtePenalizedNonparametricMean2007c} contain a similar kind of constraint on the sampling step, taking the form of a condition between the dimension of the approximation space and {\cred$\sqrt{T}$}.
}

 {\cred
 Standard techniques based on the large deviations of $X$ over the time interval $[0,T]$ only guarantee a generalization error rate of the form $T^{-\kappa}$ for some constant $\kappa \in (0,1)$, because the mixing coefficient of the joint process $(X,Y)$ is at least as large as that of $X$.
However, this bound is suboptimal for the problem at hand.
A key feature in estimating the function $\sigma$ is the interplay between two distinct types of large deviations.
One arises from the short-time movements of the Brownian motions, whereas the other is associated with the mixing covariate process $X$.
Crucially, the temporally local fluctuations of $Y$ carry more information for the estimation of $\sigma$ than the long-term trend governed by $X$.
}

{\cred 

{\cred 
Generally, large deviation estimates are utilized to control the covering number $\mathcal{N}_n$ of the model $\mathfrak{F}_n$. 
Thus, in our problem, the short-time Brownian large deviation plays a crucial role 
{\fred to derive the term $n^{-1}(\log n)\log\caln_n$, }%
which is essential for establishing the bound (\ref{20608011544}). 
This strategy is applied to the estimation of $\mathbf{R}_n^p$ to derive (\ref{20608011544}), where $\mathbf{R}_n^p$ denotes the empirical version of $\overline{\mathbf{R}}_n^p$; that is, $\mathbf{R}_n^p = E[\mathcal{E}_n^p]$ with $\mathcal{E}_n^p$ given in (\ref{202602121114}).
\footnote{This bound follows from Lemmas \ref{202602140253} and \ref{202607131330}. }
}

Consequently, we can establish the bound (\ref{20608011544}) for $\overline{\mathbf{R}}_n^p$ provided that $\overline{\mathbf{R}}_n^p$ is dominated by $\mathbf{R}_n^p$ up to a negligible term. 
Since $\overline{\mathbf{R}}_n^p$ and $\mathbf{R}_n^p$ are characterized by $(\overline{X}_{t_{j-1}})_{j=1,...,n}$ and $(X_{t_{j-1}})_{j=1,...,n}$, respectively, the large deviation associated with $X$ (and its copy $\overline{X}$) plays a primary role in this step. 
To this end, in Section \ref{202606202120}, we establish certain risk comparison theorems to compare $\overline{\mathbf{R}}_n^p$ and $\mathbf{R}_n^p$ solely through the large deviations of $X$ over $[0,T]$. 
These risk comparison theorems generally assert that, regardless of the dependent models considered, the predictable generalization error is dominated by the predictable empirical error up to a residual term, which is controlled by the underlying process.
The bound (\ref{20608011544}) for $\overline{\mathbf{R}}_n$ also follows immediately from the estimate of $\overline{\mathbf{R}}_n^p$.
}

{\cred The organization of this paper is as follows. As noted above, Section \ref{202606202120} establishes comparison results between empirical and generalization risks for arbitrary estimators. This part is general and does not rely on any specific probabilistic structure, such as a diffusion process. Furthermore, we do not assume a linear basis representation for the estimator, making our results highly applicable to neural network estimators. Section \ref{202607151441} presents the model, estimation procedure, and statistical results. Section \ref{202607151444} provides the proofs for the upper bounds on generalization and empirical errors. Finally, Section \ref{202607151446} addresses the application to deep neural network inference, presenting an upper bound in Section  \ref{202607151446a} and exploring the minimax optimality of the estimation procedure in Section \ref{202607151446b}.
}

\section{Risk comparison theorems}\label{202606202120}
In this section, we will provide basic results to bridge the generalized and empirical errors. These results are generic and hopefully  of interest independently from the diffusion process we are finally aiming at in the article.

Let $\cald$ is an open set in $\bbR^\sfd$. 
Let $\beta>0$ and $\ell$ the maximum integer satisfying $\ell<\beta$. 
Denote by $C_b^{\ell,\beta-\ell}(\cald)$ the set of $\ell$-times continuously differentiable function $f$ on $\bbR^\sfd$ such that 
$\|f\|_{\fred C_b^{\ell,\beta-\ell}(\cald)}<\infty$, where 
\bea\label{202603031601}
\|f\|_{C_b^{\ell,\beta-\ell}(\cald)}
&=&
\sum_{i:0\leq i\leq\ell}
\sup_{x\in\cald}|\partial^if(x)|+\sup_{x,y\in\cald,x\not=y}\frac{|\partial^\ell f(x)-\partial^\ell f(y)|}{|x-y|^{\beta-\ell}}. 
\eea
In (\ref{202603031601}), $\partial^if^{(i)}$ denotes the $i$-th derivative (tensor) of $f$. 
The space $C_b^{\ell,\beta-\ell}(\cald)$ is regarded as a measurable space equipped with its Borel $\sigma$-field generated by 
the topology associated with the norm $\|f\|_{C_b^{\ell,\beta-\ell}(\cald)}$.

Suppose that ${\tt K}:\bbR\to\bbR$ is a kernel function of order $\ell$, 
i.e., ${\tt K}$ is continuous and satisfies
\beas &&
\int_\bbR|x|^s|{\tt K}(x)|dx\><\>\infty{\fred\ (s=0,1,...,\ell)}, \quad
\int_{\bbR}{\tt K}(x)dx \yeq 1,\quad
\int_{\bbR}x^{s}{\tt K}(x)dx \yeq
0\ {\fred(s=1,...,\ell)}. 
\eeas
The function ${\tt K}$ can take negative values. 
An example of the kernel ${\tt K}$ of arbitrary order $\ell$ is the super kernel {\colorro ${\tt K}_{\text{super}}:\bbR\to\bbR$} defined by
\beas 
{\colorro{\tt K}_{\text{super}}(z)} &=& \frac{1}{2\pi}\int_\bbR\cos(zu)\big(1-e^{-u^{-2}}\big)du\qquad(z\in\bbR). 
\eeas
%
{\colorro 
Another kernel, which we will use later, is a symmetric kernel ${\tt K}_0\in C^\infty(\bbR;\bbR)$ such that 
its Fourier transform $\wh{{\tt K}}_0\in C^\infty(\bbR;\bbR)$ has a support in $[-1,1]$ and $\wh{{\tt K}}_0(u)=1$ for $u\in[1/2,1/2]$. ${\tt K}_0$ is also a super kernel. 
Moreover, for a fixed $m\in\bbN$, 
define ${\tt K}_1$ as ${\tt K}_1(x)={\tt K}_0(4m x)$. 
}

For a kernel ${\tt K}$ of order $\ell$, 
{\fred define} $\varphi_{d_n}:\bbR^\sfd\to\bbR$ for $d_n>0$ as 
\beas 
\varphi_{d_n}(x)
&=&
d_n^{-\sfd}\prod_{i=1}^\sfd{\tt K}\big(d_n^{-1}x_i\big)
\qquad(x=(x_i)\in\bbR^\sfd)
\eeas
{\colorro Similarly, let
\beas 
\varphi_{0,d_n}(x)
&=&
d_n^{-\sfd}\prod_{i=1}^\sfd{\tt K}_0\big(d_n^{-1}x_i\big)
\qquad(x=(x_i)\in\bbR^\sfd).
\eeas
}

%
%
{\colorb
Let $\psi:\bbR^\sfd\to\bbR$ be a bounded measurable function. 
}%
{\fred Suppose that  $\xi_{n,j}$, $\xi'_{n,j}$ $(j=1,...,n;\>n\in\bbN)$ are $\bbR^\sfd$-valued random variables. }%
{\colorro
Let 
\bea\label{202604111241}
E_n
&=&
E\bigg[n^{-1}\sum_{j=1}^nV_n(\xi'_{n,j}){\colorb \psi}(\xi'_{n,j})\bigg]
-(1+\ep)E\bigg[n^{-1}\sum_{j=1}^nV_n(\xi_{n,j}){\colorb \psi}(\xi_{n,j})\bigg]
\eea
{\fred for $\ep>0$ and a random field $V_n$ on $\bbR^\sfd$. 
}%
Define $P_n(a)$ as 
\beas 
P_n(a)
&=& 
P\bigg[
n^{-1}\sum_{j=1}^n\varphi_{d_n}(\xi'_{n,j}-a){\colorb \psi}(\xi'_{n,j})
-(1+\ep)n^{-1}\sum_{j=1}^n\varphi_{d_n}(\xi_{n,j}-a){\colorb \psi}(\xi_{n,j})>0\bigg].
\eeas
}
{\colorro
Suppose that {\colorro$\hat{\calx}$} is a measurable set in $\bbR^\sfd$. 
%
Moreover, we suppose that ${\tt B}$ is a bounded set in {\colorro$C_b^{\ell,\beta-\ell}(\bbR^\sfd)$}. 
}%
We have the following comparison theorem. 
\begin{theorem}\label{202510111529a}
Suppose that  $C_b^{\ell,\beta-\ell}(\bbR^\sfd)$-valued random variables $V_n$ $(n\in\bbN)$ satisfy
\bd
\im[(i)] $V_n\in{\tt B}$ a.s. for $n\in\bbN$. 
\im[(ii)] $V_n\geq0$ {\colorb a.s.} 
for $n\in\bbN$. 
\ed
Let $(d_n)_{n\in\bbN}$ be a sequence of positive numbers such that $d_n\to0$ as $n\to\infty$. 
Let $\ep\in(0,1]$. 
Then 
\bd
\im[(a)] For $f\in{\tt B}$ and $n\in\bbN$, 
\bea\label{202606100529}
f*\varphi_{d_n}(x)
&=&
\int_{\bbR^\sfd}f(a)\varphi_{d_n}(x-a)da
\yeq
f(x)+{\colorro \ol{\ep}_n(f,x)},
\eea
where the term ${\colorro \ol{\ep}_n(f,x)}$ 
{\colorro satisfies $\sup_{f\in{\tt B},x\in\bbR^\sfd}| \ol{\ep}_n(f,x)|=O(d_n^\beta)$ as $d_n\to0$.}

\im[(b)] 
there exists constants {\fred $C(\hat{\calx},{\tt K})$} 
and $C({\tt B},{\tt K})$ such that 
\bea\label{202510111615}
E_n
&\leq&
{\cred C(\hat{\calx},{\tt K})}{\colorro(\sup_n\|V_n\|_\infty )}{\colorb \|\psi\|_\infty}
d_n^{-\sfd}
\sup_{a\in{\colorro\hat{\calx}}}
{\colorro P_n(a)}
+
C({\tt B},{\tt K}) {\colorb \|\psi\|_\infty}d_n^\beta
\nn\\&&
{\colorro +(2+\ep)d_n^{-\sfd}\|{\tt K}\|_\infty^\sfd {\colorb \|\psi\|_\infty} \int_{(\hat{\calx})^c}V_n(a)da}
\eea
for all $n\in\bbN$. 
%

\im[(c)] {\colorro 
If additionally $\varphi_{d_n}*V_n=V_n$ a.s., then it holds that 
\bea\label{202605180801}
E_n
&\leq&
{\cred C(\hat{\calx},{\tt K})}{\colorro(\sup_n\|V_n\|_\infty )}{\colorb \|\psi\|_\infty}
d_n^{-\sfd}
\sup_{a\in{\colorro\hat{\calx}}}
{\colorro P_n(a)}
\nn\\&&
{\colorro +(2+\ep)d_n^{-\sfd}\|{\tt K}\|_\infty^\sfd 
{\colorb \|\psi\|_\infty}
\int_{(\hat{\calx})^c}V_n(a)da}
\eea
instead of (\ref{202510111615}). 
}

\ed

\end{theorem}
\proof
We start with the proof of (a) for self-containedness. 
When a function $f:\bbR^\sfd\to\bbR$ is of class $C_b^{\ell,\beta-\ell}$, with multi-index, 
we have 
\beas&&
\bigg|\int_{\bbR^\sfd}\int_0^1\sum_{{\bf k}:|{\bf k}|=\ell}\frac{(1-s)^{\ell-1}}{{\bf k}!}
f^{({\bf k})}\big(x+s(a-x)\big)ds(a-x)^{{\bf k}}ds\varphi_{d_n}(x-a)da\bigg|
\nn\\&=&
\bigg|\int_{\bbR^\sfd}\int_0^1\sum_{{\bf k}:|{\bf k}|=\ell}\frac{(1-s)^{\ell-1}}{{\bf k}!}
\big\{f^{({\bf k})}\big(x+s(a-x)\big)-f^{({\bf k})}(x)\big\}ds(a-x)^{{\bf k}}ds\varphi_{d_n}(x-a)da\bigg|
\nn\\&=&
\bigg|\int_{\bbR^\sfd}\int_0^1\sum_{{\bf k}:|{\bf k}|=\ell}\frac{(1-s)^{\ell-1}}{{\bf k}!}
\big\{f^{({\bf k})}\big(x-d_nsu\big)-f^{({\bf k})}(x)\big\}ds(-d_nu)^{{\bf k}}ds\prod_{i=1}^\sfd{\tt K}(u_i)du\bigg|
\nn\\&\leq&
C_1({\tt B},{\tt K})
d_n^\beta
\int_{\bbR^\sfd}|u|^\beta\prod_{i=1}^\sfd\big|{\tt K}(u_i)\big|du
\eeas
for all $n\in\bbN$, for some constant $C_1({\tt B},{\tt K})$ independent of $n\in\bbN$ and $f\in{\tt B}$. 
Therefore, Taylor's formula gives 
\bea\label{202510091950}&&
\int_{\bbR^\sfd}f(a)\varphi_{d_n}(x-a)da
\nn\\&=&
\int_{\bbR^\sfd}\sum_{{\bf k}:|{\bf k}|\leq\ell-1}\frac{1}{{\bf k}!}f^{({\bf k})}(x)(a-x)^{{\bf k}}\varphi_{d_n}(x-a)da
\nn\\&&
+\int_{\bbR^\sfd}\int_0^1\sum_{{\bf k}:|{\bf k}|=\ell}\frac{(1-s)^{\ell-1}}{{\bf k}!}
f^{({\bf k})}\big(x+s(a-x)\big)ds
(a-x)^{{\bf k}}ds\varphi_{d_n}(x-a)da
\nn\\&=&
f(x)+{\colorro \ol{\ep}_n(f,x)}
\eea
for all $n\in\bbN$, 
where the term ${\colorro \ol{\ep}_n(f,x)}$ is of order $O(d_n^\beta)$ uniformly in $x\in\bbR^\sfd$ and $f\in{\tt B}$, 
{\colorro and obviously ${\colorro \ol{\ep}_n(f,x)}=0$ when $\varphi_{d_n}*f=f$.}
Then 
\beas
E_n
&=&
E\bigg[n^{-1}\sum_{j=1}^nV_n(\xi'_{n,j}){\colorb \psi}(\xi'_{n,j})\bigg]
-(1+\ep)E\bigg[n^{-1}\sum_{j=1}^nV_n(\xi_{n,j}){\colorb \psi}(\xi_{n,j})\bigg]
\nn\\&=&
E\bigg[n^{-1}\sum_{j=1}^n
\int_{\bbR^\sfd}V_n(a)\varphi_{d_n}(\xi'_{n,j}-a){\colorb \psi}(\xi'_{n,j})da\bigg]
\nn\\&&
-(1+\ep)E\bigg[n^{-1}\sum_{j=1}^n
\int_{\bbR^\sfd}V_n(a)\varphi_{d_n}(\xi_{n,j}-a){\colorb \psi}(\xi_{n,j})da\bigg]
+{\colorro \ep_n}
\eeas
and hence 
\beas
E_n
&\leq&
\int_{{\colorro\hat{\calx}}}E\bigg[V_n(a)\bigg\{n^{-1}\sum_{j=1}^n
\varphi_{d_n}(\xi'_{n,j}-a){\colorb \psi}(\xi'_{n,j})
\nn\\&&\hspace{30pt}
-(1+\ep)n^{-1}\sum_{j=1}^n
\varphi_{d_n}(\xi_{n,j}-a){\colorb \psi}(\xi_{n,j})\bigg\}_+\bigg]da
\nn\\&&
{\colorro
+\int_{(\hat{\calx})^c}E\bigg[V_n(a)\bigg\{n^{-1}\sum_{j=1}^n
\varphi_{d_n}(\xi'_{n,j}-a){\colorb \psi}(\xi'_{n,j})
}
\nn\\&&\hspace{30pt}
{\colorro
-(1+\ep)n^{-1}\sum_{j=1}^n
\varphi_{d_n}(\xi_{n,j}-a){\colorb \psi}(\xi_{n,j})\bigg\}_+\bigg]da}
+{\colorro \ep_n}
\eeas
since $V_n\geq0$, 
{\colorro where $\ep_n{\cred \leq C({\tt B},{\tt K}){\colorb \|\psi\|_\infty}d_n^\beta}$ in Case (b) and $\ep_n=0$ in Case (c).} 
Therefore
\beas
E_n
&\leq&
{\fred(2+\ep)}
(\sup_n\|V_n\|_\infty )d_n^{-\sfd}\|{\tt K}\|_\infty^\sfd
\text{Leb}({\colorro{\hat{\calx}}}){\colorb \|\psi\|_\infty}
\nn\\&&
\times
\sup_{a\in{\colorro\hat{\calx}}}
P\bigg[n^{-1}\sum_{j=1}^n
\varphi_{d_n}(\xi'_{n,j}-a){\colorb \psi}(\xi'_{n,j})
-(1+\ep)n^{-1}\sum_{j=1}^n
\varphi_{d_n}(\xi_{n,j}-a){\colorb \psi}(\xi_{n,j})>0\bigg]
\nn\\&&
{\colorro +(2+\ep)d_n^{-\sfd}\|{\tt K}\|_\infty^\sfd 
{\colorb \|\psi\|_\infty}
\int_{(\hat{\calx})^c}V_n(a)da}
+{\colorro \ep_n}. 
\eeas
This shows (\ref{202510111615}) {\colorro and (\ref{202605180801}).}
\qed\halflineskip

{\colorro Let 
\beas 
Q_n(a) 
&=& 
P\bigg[
\bigg|n^{-1}\sum_{j=1}^n\bigg\{\varphi_{d_n}(\xi_{n,j}-a){\colorb \psi}(\xi_{n,j})
-E\big[\varphi_{d_n}(\xi_{n,j}-a){\colorb \psi}(\xi_{n,j})\big]\bigg\}\bigg|
\nn\\&&\hspace{60pt}
{\cred\geq\>}\frac{\ep}{2(1+\ep)}n^{-1}\sum_{j=1}^nE\big[\varphi_{d_n}(\xi_{n,j}-a){\colorb \psi}(\xi_{n,j})\big]
\bigg]
\eeas
}
Theorem \ref{202510111529a} is customized for the generalization error as the following comparison theorem. 
\begin{theorem}\label{202510111529}
Additionally to the conditions in Theorem \ref{202510111529a}, suppose that 
$(\xi_{n,j})_{j=1,...,n}=^d(\xi'_{n,j})_{j=1,...,n}$ for all $n\in\bbN$. Then, for $E_n$ of (\ref{202604111241}), 
\bd
\im[(a)] 
it holds that 
\bea\label{202510112259}
E_n
&\leq&
2{\cred C(\hat{\calx},{\tt K})}{\colorro(\sup_n\|V_n\|_\infty )}{\colorb \|\psi\|_\infty}
d_n^{-\sfd}
\sup_{a\in{\colorro\hat{\calx}}}
{\colorro Q_n(a)}
+C({\tt B},{\tt K}) {\colorb \|\psi\|_\infty}d_n^\beta
\nn\\&&
{\colorro +(2+\ep)d_n^{-\sfd}\|{\tt K}\|_\infty^\sfd 
{\colorb \|\psi\|_\infty}
\int_{(\hat{\calx})^c}V_n(a)da}
\eea
for all $n\in\bbN$. 
The two constants in (\ref{202510112259}) are the same as those of (\ref{202510111615}), respectively. 

\im[(b)] 
{\colorro 
If additionally $\varphi_{d_n}*V_n=V_n$ a.s., then 
\bea\label{202605180826}
E_n
&\leq&
2{\cred C(\hat{\calx},{\tt K})}(\sup_n\|V_n\|_\infty ){\colorb \|\psi\|_\infty}
d_n^{-\sfd}
\sup_{a\in{\colorro\hat{\calx}}}
{\colorro Q_n(a)}
\nn\\&&
{\colorro +(2+\ep)d_n^{-\sfd}\|{\tt K}\|_\infty^\sfd{\colorb \|\psi\|_\infty} \int_{(\hat{\calx})^c}V_n(a)da}
\eea
instead of (\ref{202510112259}). 
}
\ed
\end{theorem}
\proof 
For notational simplicity, write 
\beas 
{\tt M}=n^{-1}\sum_{j=1}^n\varphi_{d_n}(\xi_{n,j}-a){\colorb \psi}(\xi_{n,j})\quad\text{and}\quad
{\tt M}'=n^{-1}\sum_{j=1}^n\varphi_{d_n}(\xi'_{n,j}-a){\colorb \psi}(\xi'_{n,j}). 
\eeas
Both ${\tt M}$ and ${\tt M}'$ are depending on $a$. 
We have $E[{\tt M}]=E[{\tt M}']$ and 
\beas &&
{\cred P_n(a)\yeq}
P\big[{\tt M}'-(1+\ep){\tt M}>0\big]
\nn\\&\leq&
P\bigg[{\tt M}'-E[{\tt M}']>\frac{\ep}{2(1+\ep)}E[{\tt M}]\bigg]
+P\bigg[{\tt M}'-E[{\tt M}']\leq\frac{\ep}{2(1+\ep)}E[{\tt M}], \>{\tt M}'-(1+\ep){\tt M}>0\bigg]
\nn\\&\leq&
P\bigg[{\tt M}'-E[{\tt M}']>\frac{\ep}{2(1+\ep)}E[{\tt M}]\bigg]
+P\bigg[(1+\ep){\tt M}-E[{\tt M}']\leq\frac{\ep}{2(1+\ep)}E[{\tt M}]\bigg]
\nn\\&\leq&
P\bigg[{\tt M}-E[{\tt M}]>\frac{\ep}{2(1+\ep)}E[{\tt M}]\bigg]
+P\bigg[{\tt M}-E[{\tt M}]\leq-\frac{\ep(1+2\ep)}{2(1+\ep)^2}E[{\tt M}]\bigg].
\eeas
Due to (\ref{202510111615}), this gives (\ref{202510112259}) when $E[{\tt M}]>0$ for all $a\in\calx$. 
Obviously, 
\beas
E_n
&\leq& 
{\cred C(\hat{\calx},{\tt K})(\sup_n\|V_n\|_\infty )}
{\colorb \|\psi\|_\infty}d_n^{-\sfd}+C({\tt B},{\tt K}){\colorb \|\psi\|_\infty} d_n^\beta
{\colorro +(2+\ep)d_n^{-\sfd}\|{\tt K}\|_\infty^\sfd {\colorb \|\psi\|_\infty}\int_{(\hat{\calx})^c}V_n(a)da}
\eeas
from (\ref{202510111615}), therefore 
(\ref{202510112259}) holds when $E[{\tt M}]\leq0$ for some $a\in{\colorro\hat{\calx}}$
{\colorro since $Q_n(a)=1$}. Thus (\ref{202510112259}) is valid in any case. 
\qed\halflineskip

{\colorro 
Suppose that $\calx$ is a compact set in $\bbR^\sfd$ 
and $\hat{\calx}$ is an open set of $\bbR^\sfd$ such that $\calx{\colorro\subset\hat{\calx}}$. 
Let $m\in\bbN$. 
As before, suppose that  $\xi_n=(\xi_{n,j})_{j=1,...,n}$ and $\xi_n'=(\xi_{n,j}')_{j=1,...,n}$ $(n\in\bbN)$ are 
collections of $\bbR^\sfd$-valued random variables; 
these variables may be correlated with each other. 
The sequence of positive numbers $(d_n)_{n\in\bbN}$ satisfies $\lim_{n\to\infty}d_n=0$. 
%
{\cred Write $[\sfd_D]=\{1,...,\sfd_D\}$. 
Suppose that $D_n=(D_n^{(i)})_{i\in[\sfd_D]}$ is a collection of ${\tt B}$-valued random fields $D_n^{(i)}$ }%
satisfying that $\text{supp }{\cred D_n^{(i)}}\in\calx$ a.s. for every $n\in\bbN$ {\cred and $i\in[\sfd_D]$}. 
We define $\bbR_n(\xi_n)$ as 
\beas 
\bbR_n(\xi_n)
&=& 
E\bigg[n^{-1}\sum_{j=1}^n\big|D_n(\xi_{n,j})\big|^{2m}{\colorb\psi(\xi_{n,j})}\bigg].
\eeas
For $\ep>0$, let 
\bea\label{202606121043}
Q_{1,n}(a) 
&=& 
P\bigg[
\bigg|n^{-1}\sum_{j=1}^n\bigg\{\varphi_{0,d_n/(4m)}(\xi_{n,j}-a){\colorb \psi}(\xi_{n,j})
-E\big[\varphi_{0,d_n/(4m)}(\xi_{n,j}-a){\colorb \psi}(\xi_{n,j})\big]\bigg\}\bigg|
\nn\\&&\hspace{60pt}
{\cred\geq\>}\frac{\ep}{2(1+\ep)}n^{-1}\sum_{j=1}^nE\big[\varphi_{0,d_n/(4m)}(\xi_{n,j}-a){\colorb \psi}(\xi_{n,j})\big]
\bigg]. 
\eea 

We obtain the following comparison theorem for the $L^{2m}$-risks. 
{\fred 
This theorem allows us to replace the time series training data with an independent copy in the risk function.
}
\begin{theorem}\label{202606111044}
{\cred Suppose that $\psi:\bbR^\sfd\to\bbR$ is a nonnegative bounded measurable function. }%
Suppose that $\xi_n=^d\xi_n'$ for every $n\in\bbN$. 
Then there exists a constant $C$ such that 
\bea\label{202606121049}
\bbR_n(\xi_n')
&\leq& 
C
\bigg[
\bbR_n(\xi_n)
+d_n^{2m\beta}
+d_n^{-\sfd}
\sup_{a\in\hat{\calx}}Q_{1,n}(a)
\bigg]
\eea
for all $n\in\bbN$. 
\end{theorem}
\proof 
We define $\bbR_{n,d_n}(\xi_n)$ as 
\beas 
\bbR_{n,d_n}(\xi_n)
&=& 
E\bigg[n^{-1}\sum_{j=1}^n\big|\big(\varphi_{0,d_n}*D_n\big)(\xi_{n,j})\big|^{2m}
{\cred \psi(\xi_{n,j})}\bigg]
\eeas

For some constant $C(m)$ depending on $m$, 
\beas 
|x|^{2m} &\leq& C(m)\big(|y|^{2m}+|x-y|^{2m}\big)
\eeas
for all $x,y\in{\cred\bbR^{\sfd_D}}$. 
Therefore, 
\bea\label{202606100824a}
\bbR_n(\xi_n')
&\simleq&
\bbR_{n,d_n}(\xi_n')
+\big(\|D_n-\varphi_{0,d_n}*D_n\|_{L^\infty(\Omega\times\bbR^\sfd{\cred\times[\sfd_D]})}\big)^{2m}
\eea
and 
\bea\label{202606100824}
\bbR_{n,d_n}(\xi_n)
&\simleq&
\bbR_n(\xi_n)
+\big(\|D_n-\varphi_{0,d_n}*D_n\|_{L^\infty(\Omega\times\bbR^\sfd{\cred\times[\sfd_D]})}\big)^{2m}. 
\eea
{\cred Here the non-negativity of $\psi$ is used, and 
the constant in $\simleq$ depends on $\|\psi\|_\infty$. }

For a function $g:\bbR^\sfd\to\bbC$, define $g^\sim:\bbR^\sfd\to\bbC$ as $g^\sim(u)={\fred \ol{g(-u)}}$ for $u\in\bbR^\sfd$. 
Now, {\cred in place of $V_n$, we consider the random function }
\bea\label{202606110815}
{\cred V_{1,n}}(x) &=& \big|\big(\varphi_{0,d_n}*D_n\big)(x)\big|^{2m}
\eea
for $x\in\bbR^\sfd$; $D_n$ is compactly supported, while ${\cred V_{1,n}}$ is not so 
because of the Paley-Wiener theorem. 
%
We denote by ${\mathfrak F}$ and ${\mathfrak F}^{-1}$ the Fourier transform and the Fourier inversion, respectively. 
Then, for $D_n=(D_n^{(i)})_{i\in[\sfd_D]}$, we have 
\beas 
{\cred V_{1,n}}(x) 
&=&
\big|{\mathfrak F}^{-1}{\mathfrak F}\big(\varphi_{0,d_n}*D_n\big)\big|^{2m}(x)
\nn\\&=&
\sum_{i_1,...,i_m}^{\sfd_D}\prod_{k=1}^m
\big|{\mathfrak F}^{-1}{\mathfrak F}\big(\varphi_{0,d_n}*D_n^{(i_k)}\big)\big|^{2}(x)
\nn\\&=&
\sum_{i_1,...,i_m}^{\sfd_D}\prod_{k=1}^m
\big[{\mathfrak F}^{-1}\big\{{\mathfrak F}\big(\varphi_{0,d_n}*D_n^{(i_k)}\big)*\big({\mathfrak F}\big(\varphi_{0,d_n}*D_n^{(i_k)}\big)\big)^\sim\big\}(x)\big]
\nn\\&=&
\sum_{i_1,...,i_m}^{\sfd_D}
{\mathfrak F}^{-1}\bigg[\big\{{\mathfrak F}\big(\varphi_{0,d_n}*D_n^{(i_1)}\big)*\big({\mathfrak F}\big(\varphi_{0,d_n}*D_n^{(i_1)}\big)\big)^\sim\big\}
*\cdots
\nn\\&&\hspace{30pt}
\cdots
*\big\{{\mathfrak F}\big(\varphi_{0,d_n}*D_n^{(i_m)}\big)*\big({\mathfrak F}\big(\varphi_{0,d_n}*D_n^{(i_m)}\big)\big)^\sim\big\}
\bigg](x),
\eeas
therefore, 
the Fourier transform of $V_{\fred 1,n}$ is 
\beas 
\wh{{\cred V_{1,n}}}(u)
&=&
\sum_{i_1,...,i_m}^{\sfd_D}
\big\{{\mathfrak F}\big(\varphi_{0,d_n}*D_n^{(i_1)}\big)*\big({\mathfrak F}\big(\varphi_{0,d_n}*D_n^{(i_1)}\big)\big)^\sim\big\}
*\cdots
\nn\\&&\hspace{30pt}
\cdots
*\big\{{\mathfrak F}\big(\varphi_{0,d_n}*D_n^{(i_m)}\big)*\big({\mathfrak F}\big(\varphi_{0,d_n}*D_n^{(i_m)}\big)\big)^\sim\big\} (u)
\nn\\&=&
\sum_{i_1,...,i_m}^{\sfd_D}
\big(\wh{\varphi_{0,d_n}}(u)\wh{D_n^{(i_1)}}(u)\big)*\big(\wh{\varphi_{0,d_n}}(-u)\wh{D_n^{(i_1)}}(-u)\big)
*\cdots
\nn\\&&\hspace{30pt}
\cdots
*\big(\wh{\varphi_{0,d_n}}(u)\wh{D_n^{(i_m)}}(u)\big)*\big(\wh{\varphi_{0,d_n}}(-u)\wh{D_n^{(i_m)}}(-u)\big)
\eeas
for $u=(u_i)_{i=1,...,\sfd}\in\bbR^\sfd$. 
{\fred Here the hat $\>\wh{\ }\>$ denotes the Fourier transform. }%
Moreover, $\text{supp }\wh{\varphi_{0,d_n}}\subset[-d_n^{-1},d_n^{-1}]^\sfd$ since 
\beas 
\wh{\varphi_{0,d_n}}(u)
&=& 
\prod_{i=1}^\sfd\wh{{\tt K}}_0(d_nu_i). 
\eeas
So, $\text{supp }\wh{V_{\fred 1,n}}\in[-2md_n^{-1},2md_n]^\sfd$. 
On the other hand, 
$
\wh{\varphi_{0,d_n/(4m)}}(u)=1
$ 
for $u\in[-2md_n^{-1},2md_n^{-1}]^\sfd$, due to the definition of ${\tt K}_0$. 
Thus, we see 
\beas 
\wh{\varphi_{0,d_n/(4m)}}(u)\wh{{\cred V_{1,n}}}(u)
&=&
\wh{{\cred V_{1,n}}}(u)
\eeas
for all $u\in\bbR^\sfd$. 
Consequently, 
\bea\label{202606100955}
\varphi_{0,d_n/(4m)}*{\cred V_{1,n}} &=& {\cred V_{1,n}}. 
\eea
Then we obtain Inequality 
(\ref{202605180826}) of Theorem \ref{202510111529}
{\fred with $V_n$ replaced by $V_{1,n}$ }%
given by (\ref{202606110815}), the kernel ${\tt K}_1$ for ${\tt K}$, and 
{\cred 
\beas 
\check{{\tt B}}
&=&
\bigg\{\bigg(\sum_{i=1}^{\sfd_D}\big(f^{(i)}\big)^2\bigg)^m;\>f^{(i)}\in\check{\tt B}_0\ (i\in[\sfd_D])\bigg\}
\eeas
in place of the original ${\tt B}$, where 
\beas 
{\fred\check{
{\tt B}}}_0 
&=& 
\bigg\{f\in C_b^{\ell,\beta-\ell}(\bbR^\sfd);\>\|f\|_{C_b^{\ell,\beta-\ell}(\bbR^\sfd)}\leq
\bigg(\int_\bbR\big|{\tt K}_0(y)\big|dy\bigg)^\sfd
{\>\fred\sup_{g\in{\tt B}}\|g\|_{C^{\ell,\beta-\ell}_b(\bbR^\sfd)}}
\bigg\}.
\eeas
The set $\check{{\tt B}}$ is a bounded subset of $C_b^{\ell,\beta-\ell}(\bbR^\sfd)$. 
}
More precisely, for $V_{1,n}(x):= \big|\big(\varphi_{0,d_n}*D_n\big)(x)\big|^{2m}$, it holds that 
\bea\label{202606120841}
E_{1,n}
&\leq&
2{\cred C(\hat{\calx},{\tt K}_1)}(\sup_n\|V_{1,n}\|_\infty ){\colorb \|\psi\|_\infty}
d_n^{-\sfd}
\sup_{a\in{\colorro\hat{\calx}}}
{\colorro Q_{1,n}(a)}
\nn\\&&
{\colorro +(2+\ep)d_n^{-\sfd}\|{\tt K}_1\|_\infty^\sfd {\colorb \|\psi\|_\infty}\int_{(\hat{\calx})^c}V_{1,n}(a)da},
\eea
where 
\beas 
E_{1,n}
&=&
\bbR_{n,d_n}(\xi_n')
-(1+\ep)\bbR_{n,d_n}(\xi_n)
\eeas
and 
$Q_{1,n}(a)$ is given by (\ref{202606121043}).

%
%
%
We have 
\beas &&
\int_{a\in(\hat{\calx})^c}V_{1,n}(a)da
\yeq
\int_{a\in(\hat{\calx})^c}
\big|\big(\varphi_{0,d_n}*D_n\big)(a)\big|^{2m}da
\nn\\&=& 
\int_{a\in(\hat{\calx})^c}
\bigg|
\int_\calx\varphi_{0,d_n}(a-x)D_n(x)dx\bigg|^{2m}da
\quad{\cred (\text{supp }D_n^{(i)}\in\calx\text{ a.s.})}
\nn\\&\leq&
\int_{a\in(\hat{\calx})^c}
\bigg(\int_{\calx}\big|\varphi_{0,d_n}(a-x)\big|^2dx\bigg)^m
\bigg(\int_{\calx}\big|D_n(x)\big|^{{\cred2}}dx\bigg)^mda
\nn\\&\leq&
C({\tt B},m{\fred, \sfd_D})\text{Leb}(\calx)^{{\fred2}m-1}
\int_{a\in(\hat{\calx})^c}
\int_{\calx}\big|\varphi_{0,d_n}(a-x)\big|^{2m}dxda
\nn\\&\leq&
C({\tt B},m{\fred, \sfd_D})\text{Leb}(\calx)^{{\fred2}m-1} \sup_{a\in(\hat{\calx})^c,\>x\in\calx}\big|\varphi_{0,d_n}(a-x)\big|^{2m-1}
\int_{\calx}\int_{a\in(\hat{\calx})^c}\big|\varphi_{0,d_n}(a-x)\big|dadx
\nn\\&\leq&
C({\tt B},m{\fred, \sfd_D})\text{Leb}(\calx)^{{\fred2}m} \sup_{a\in(\hat{\calx})^c,\>x\in\calx}\big|\varphi_{0,d_n}(a-x)\big|^{2m-1}
\int_{y=(y_i)\in\bbR^\sfd}\prod_{i=1}^\sfd|{\tt K}_0(y_i)|dy
\eeas
Since ${\tt K}_0$ is rapidly decreasing, 
$\int_{y=(y_i)\in\bbR^\sfd}\prod_{i=1}^\sfd|{\tt K}_0(y_i)|dy<\infty$, besides, 
for any $L>0$, there exists a constant $C({\tt K}_0,L)$ such that 
\beas
\sup_{a\in(\hat{\calx})^c,\>x\in\calx}\big|\varphi_{0,d_n}(a-x)\big|
&\leq&
C({\tt K}_0,L,\calx,\hat{\calx})d_n^L
\eeas
For {\cred$V_{1,n}$} of (\ref{202606110815}), we conclude that 
there exists a constant $C({\tt B}, m, {\tt K}_0,\calx,\hat{\calx}{\fred, \sfd_D})$ such that 
\bea\label{202606121051}
\int_{(\hat{\calx})^c}V_{1,n}(a)da
&\leq&
C({\tt B}, m, {\tt K}_0,\calx,\hat{\calx}{\fred, \sfd_D})d_n^{2m\beta+\sfd}
\eea
for all $n\in\bbN$. 

Theorem \ref{202510111529a} (a) gives the estimate 
\bea\label{202606121057}
\|D_n-\varphi_{0,d_n}*D_n\|_{L^\infty(\Omega\times\bbR^\sfd{\cred\times[\sfd_D]})}
&=& 
O(d_n^\beta)
\eea
since {\cred each $D_n^{(i)}$} takes values in ${\tt B}$ a.s. 

Now we obtain (\ref{202606121049}) from 
{\cred(\ref{202606100824a})}, (\ref{202606100824}), (\ref{202606120841}), (\ref{202606121051}) and (\ref{202606121057}). 
\qed\halflineskip

Later, we will use $D_n(x)$ for the components of $\big(\wh{S}_n(x)-S^*(x)\big)\calw(x)^{1/2}$.

}

\section{Learning of the diffusion matrix}\label{202607151441}
This and the following sections will discuss the problem of evaluating the generalization error in learning the diffusion {\cred matrix}.

\subsection{Stochastic regression model with an $\alpha$-mixing covariate process}
Let $\bbR_+=\bbR_{\geq0}=[0,\infty)$. 
On a stochastic basis $(\Omega,\calf,\bbF,P)$ with a right-continuous filtration $\bbF=(\calf_t)_{t\in\bbR_+}$, 
we consider an $\sfm$-dimensional adapted process $Y=(Y_t)_{t\in\bbR_+}$ that satisfies 
\bea\label{202509191103}
dY_t &=& b(X_t)dt+\sigma(X_t)dw_t\quad(t\in\bbR_+)
\eea
for a $\sfd$-dimensional {\cred $\bbF$-progressively measurable} process $X=(X_t)_{t\in\bbR_+}$. 
The coefficients functions $b:\bbR^\sfd\to\bbR^\sfm$ and $\sigma:\bbR^\sfd\to\bbR^\sfm\otimes\bbR^\sfr$ 
are {\colorro measurable}, and $w=(w_t)_{t\in\bbR_+}$ is an $\sfr$-dimensional Wiener process. 
{\colorro
More precisely, we assume that the map $t\mapsto b(X_t)$ is locally integrable in $dt$ a.s., and that the map $t\mapsto\sigma(X_t)$ is locally square-integrable in $dt$ a.s. }%
The model (\ref{202509191103}) is called a stochastic regression model. 
When $X=Y$ in particular, $Y$ is a diffusion process, though the model (\ref{202509191103}) 
admits more general processes {\cred as} $X$. 
In a statistical context, the process $X$ is a covariate process. 
We assume that $X$ is {\cred geometrically} $\alpha$-mixing. More precisely, 
for the $\sigma$-field $\calb_I^X=\sigma[X_t,;\>t\in I]$ for $I\subset\bbR_+$, 
the $\alpha$-mixing coefficient is given by 
\beas 
\alpha^X(r)
&=& 
\sup_{t\in\bbR_+}\sup\bigg\{\big|P[{\tt A}\cap{\tt B}]-P[{\tt A}]P[{\tt B}]\big|;\>
{\tt A}\in\calb^X_{[0,t]},{\tt B}\in\calb^X_{[t+r,\infty)}\bigg\}. 
\eeas
{\cred We will assume $\alpha^X(r)$ decays exponentially fast as $r\to\infty$. }%

{\colorro 
We assume that the process $X$ is an It\^o semimartingale having a differential representation 
\bea\label{202608011934}
dX_t &=& {\cred b^X_t}dt+{\cred\sigma^X_t}dw_t^X, 
\eea
where 
and $w^X=(w^X_t)_{t\in\bbR_+}$ is an $\sfr_X$-dimensional $\bbF$-Wiener process (possiblly correlated with $w$), 
${\cred b^X}=({\cred b^X_t})_{t\in\bbR_+}$ and ${\cred\sigma^X}=({\cred\sigma^X_t})_{t\in\bbR_+}$ are 
$\bbR^\sfd$-valued and $\bbR^\sfd\otimes\bbR^{\sfr_X}$-valued $\bbF$-progressively measurable processes, respectively. 

Let $\cals$ [resp. $\cals_{+}$] be the 
set of $\sfm\times\sfm$ symmetric matrices 
[resp. {\colorb nonnegative}-definite symmetric matrices]. 
For a measurable set $\calx$ in $\bbR^\sfd$, denote by $\bbS$ a family of measurable functions $S:{\colorb\calx}\to\cals_+$ 
such that 
\bea\label{202505050251}
0<\inf_{x\in\calx,S\in\bbS}\lambda_{\min}(S(x))
\leq\sup_{x\in\calx,S\in\bbS}\lambda_{\max}(S(x))<\infty, 
\eea
where $\lambda_{\min}(S)$ and $\lambda_{\max}(S)$ are the minimum eigenvalue and the maximum eigenvalue of the symmetric matrix $S$, respectively.

Consider a set $\sfB \subset C(\calx;\cals_+)\cap\bbS$ such that 
$\{S|_{\calx^o};\>S\in\sfB\}$ is a bounded set in $ C^{\ell,\beta-\ell}_b(\calx^o;\bbR^\sfm\otimes\bbR^\sfm)$. 
%
Here $\calx^o=\text{Int }\calx$ and 
the set $C^{\ell,\beta}_b(\calx^o;\bbR^\sfm\otimes\bbR^\sfm)$ is identified with $\big(C^{\ell,\beta}_b(\calx^o)\big)^{\sfm^2}$. 
Since $C^{\ell,\beta}_b(\calx^o;\bbR^\sfm\otimes\bbR^\sfm)$ has a topology associated with the norm 
$\|\cdot\|_{C^{\ell,\beta-\ell}(\calx^o)}$, it is regarded as a measurable space.

Suppose that $\check{\calx}$ is an open set of $\bbR^\sfd$. 
Denote by $S^*:\bbR^\sfd\to\cals_+$ the true function of {\cred$S=\sigma\sigma^\star$, $\star$ denoting the matrix transpose}. 
{\fred For $n\in\bbN$, let $\bbI_{0,n}=\{0,1,...,n\}$ and $\bbI_n=\{1,...,n\}$. }%
We will often identify $S:\bbR^\sfd\to\cals_+$ with the restriction $S|_{\sf A}$ of $S$ to a set ${\sf A}\subset\bbR^\sfd$. 
We assume the following condition.  
}%
\bd
\im[{\bf[A1]}] 
\bd
\im[(i)] 
There exists a positive constant $\gamma$ such that 
$\alpha^X(r)\leq\gamma^{-1}e^{-\gamma r}$ for all $r>0$. 

\im[(ii)] 
{\colorro 
For every $p>1$, 
\bea\label{202606291546}
\sup_{t\in\bbR_+}\big(\big\||{\cred b^X_t}|1_{\{X_t\in\check{\calx}\}}\big\|_p+\big\||{\cred\sigma^X_t}|1_{\{X_t\in\check{\calx}\}}\big\|_p\big) &<& \infty. 
\eea
}

\im[(iii)] {\colorro 
$S^*|_{\calx}\in\sfB$, {\cred $\sigma^*\in C(\check{\calx};\bbR^\sfm\otimes\bbR^\sfr)$, }
$S^*|_{\check{\calx}}\in C^2(\check{\calx};\bbR^\sfm\otimes\bbR^\sfm)$, 
and $b^*|_{\check{\calx}}\in C^2(\check{\calx};\bbR^\sfm)$. }

\im[(iv)] 
{\colorro
There exists $p_0>2$ such that 
{\cred 
\bea\label{202608061252}
{\fred\sup_{n\in\bbN}}
\sup_{j\in\bbI_n}\big\|h_n^{-1/2}|Y_\tj-Y_\tjm|\big\|_{p_0} &<& \infty.
\eea
}
}
\ed
\ed

{\cred Obviously, (\ref{202608061252}) holds if 
\bea\label{202608061257}
{\fred\sup_{s>0}}
\sup_{t\in\bbR_+}\big\|s^{-1/2}|Y_{t+s}-Y_t|\big\|_{p_0} &<& \infty.
\eea
A sufficient condition for (\ref{202608061257}), and thus for $[A1]$ (iv), is that 
}
\bd
\im[(iv$'$)] There exists $p_0>2$ such that 
$\sup_{t\in\bbR_+}\big(\big||b(X_t)|\big\|_{p_0}+\big\||\sigma(X_t)|\big\|_{p_0})<\infty$. 
\ed

The geometric mixing property like $[A1]$ (i) is well established for various stochastic processes including diffusion processes. 
We do not assume stationarity of $X$.

%
We will consider learning of the $\bbR^\sfm\otimes\bbR^\sfm$-valued function $S$ based on the data $(X_\tj,Y_\tj)_{t\in\bbI_{0,n}}$ 
sampled from $Z=(X,Y)$, 
where $\tj=t^n_j=jh$, $h=h_n$. 
High-frequency and long-run sampling is considered, that is, 
${\cred T=T_n=\>}nh\to\infty$ and $h\to0$ as $n\to\infty$. 
\halflineskip

{\colorb
\begin{remark}\rm
Let 
$\calb_I^{X,dY}=\sigma[X_t,Y_t-Y_s;\>s,t\in I]$ for $I\subset\bbR_+$, and 
\beas 
\alpha^{X,dY}(r)
&=&
\sup_{t\in\bbR_+}\sup\bigg\{\big|P[{\tt A}\cap{\tt B}]-P[{\tt A}]P[{\tt B}]\big|;\>
{\tt A}\in\calb^{X,dY}_{[0,t]},{\tt B}\in\calb^{X,dY}_{[t+r,\infty)}\bigg\}. 
\eeas
Then $\alpha^{X,dY}(r)\geq\alpha^X(r)$. 
This suggests that it is impossible to obtain the correct rate of 
convergence of the generalization error if the mixing condition is 
applied in a conventional way. Under this conventional approach, 
an error bound $T^{-\kappa}$ would follow for $T = nh$ and some 
constant $\kappa < 1$. This is because the mixing coefficient is 
expressed in terms of time, not $n$, making the resulting bound 
suboptimal since $T \ll n$. The comparison theorems in Section  \ref{202606202120}
bypass this difficulty.
{\cred Incidentally, the bound takes the form of $T^{-\kappa}$ 
in estimation of the drift function. }
\end{remark}
}

{\colorr 
\subsection{Two-step estimator for $S$}
{\cred Let $\beta\in\bbR_{\geq2}$. }%
Fix a continuous function 
$\onecalw:\bbR^\sfd\to\bbR_+=[0,\infty)$ such that $\onecalw^{{\colorro 1/2}}\in C^{\ell,\beta-\ell}_b(\bbR^\sfd)$ and 
$
\text{supp }\onecalw
\subset{\colorro\calx}
$ 
.
%

{\cred Let ${\tt M}_1[{\tt M}_2]=\text{Tr }({\tt M}_1{\tt M_2}^\star)$ for matrices ${\tt M}_1$ and ${\tt M}_2$ of the same size. }%
We write $y^{\otimes2}$ for $yy^\star$ for a $\sfm$-dimensional column vector $y$. 
Consider the contrast function 
\bea\label{202505042250}
\Phi_n (S,Z) &=& {\colorb\frac{1}{2n}}\sum_{j\in\bbI_n}\big\{h^{-1}S(X_\tjm)^{-1}[(\Delta_jY)^{\otimes2}]+\log\det S(X_\tjm)\big\}\onecalw(X_\tjm)
\eea
for $S\in\bbS$, 
with $\Delta_jY=Y_\tj-Y_\tjm$. 

%
For $n\in\bbN$, let ${\mathfrak F}_n$ be a family of mappings such that 
\bea\label{202604101836}
{\mathfrak F}_n\subset\bbS. 
\eea 
A family ${\mathfrak F}_n$ satisfying (\ref{202604101836}) is easily constructed, for example, 
as the image of a bounded set of 
{\cred $\cals$-valued continuous functions on $\calx$} 
by the exponential map.

The norm $|S|$ is defined as $|S|=\big(\sum_{i,j}S_{i,j}^2\big)^{1/2}$ for a matrix $S=(S_{i,j})$. 
We will use the following notation: 
{\colorb 
\beas 
|S|_n
&=& 
\bigg(
n^{-1}\sum_{j\in\bbI_n}\big|
S(X_\tjm)\big|^2\calw(X_\tjm)\bigg)^{1/2}
\eeas
for a function $S:\calx\to\bbR^\sfm\otimes\bbR^\sfm$. 
{\cred The random variable} 
$|S|_n$ is well defined; $|S(x)|\calw(x)$ 
{\cred is understood to be zero whenever} $\calw(x)=0$. 
We often follow this convention. 
}

Our estimator is constructed in two steps.
The initial estimator $\wh{S}_n^\DL$ is constructed to be 
$\sigma[(X_\tj,Y_\tj)_{j\in\bbI_{0,n}}]$-measurable and to asymptotically minimizes $\Phi_n (S,Z)$ over $S\in{\mathfrak F}_n$. 
Define $\Delta _n$ as 
\bea\label{202602140236}
\Delta_n
&=& 
E\bigg[\Phi_n (\wh{S}_n^\DL;Z)-\inf_{S\in{\mathfrak F}_n }\Phi_n (S;Z)\bigg].
\eea
In practice, the optimization error $\Delta_n$ needs to be kept small.

The second-stage estimator 
$\wh{S}_n$ is a $\sfB$-valued measurable map satisfying 
{\colorb 
\beas
\wh{S}_n &\in& \sfB\quad a.s.
\eeas
}%
$\wh{S}_n$ is a projection of $\wh{S}_n^\DL$ to the set $\sfB$. 
{\colorb 
The error in this procedure is denoted by 
\bea\label{202606251832}
\wh{\Delta}_n
&=&
{\cred E\bigg[}
\big|\wh{S}_n-\wh{S}^\DL_n\big|_n^2
-\inf_{S\in\sfB}\big|S-\wh{S}^\DL_n\big|_n^2{\cred\bigg]}. 
\eea

}

\subsection{Generalization error}

Denote by $\ol{Z}=(\ol{X},\ol{Y})$ an independent copy of $Z=(X,Y)$. 
The generalization error $\ol{\bfR}_n=\ol{\bfR}_n(\wh{S}_n){\fred\>=\ol{\bfR}_n(\wh{S}_n,S^*)}$ for $\wh{S}_n$ is defined as 
$
\ol{\bfR}_n
=
{\colorro E\big[\Phi_n(\wh{S}_n,\ol{Z})-\Phi_n(S^*,\ol{Z})\big]}
$, that is, 
\bea\label{202603120055}
\ol{\bfR}_n
&=&
 \frac{1}{2n} E\bigg[\sum_{j\in\bbI_n}\bigg\{(\wh{S}_n(\ol{X}_\tjm)^{-1}-S^*(\ol{X}_\tjm)^{-1})
 \big[h^{-1}(\Delta_j\ol{Y})^{\otimes2}\big]+\log\frac{\det \wh{S}_n(\ol{X}_\tjm)}{\det S^*(\ol{X}_\tjm)}\bigg\}
\onecalw(\ol{X}_\tjm)\bigg].
\nn\\&&
\eea
We do not assume the stationarity of $(X,Y)$. 

The discrepancy between two elements $S$ and $S^*$ of $\bbS$ is evaluated by the loss function 
\bea\label{202608011524}
U (x,S)
&=& 
\half\bigg\{\big(S(x)^{-1}-S^*(x)^{-1}\big)[S^*(x)]+\log\frac{\det S(x)}{\det S^*(x)}\bigg\}\onecalw(x) \quad (x\in\calx). 
\eea
The function $U(\cdot,S)$ is extended as $U(x,S)=0$ for $x\in\calx^c$. 
By definition, $U(x,S)\geq0$. 

{\colorro
The non-negativity of the generalization error $\ol{\bfR}_n=\ol{\bfR}_n(\wh{S}_n)$ is not guaranteed. %
In this sense, a natural risk is 
the predictable generalization error $\ol{\bfR}_n^p=\ol{\bfR}_n^p(\wh{S}_n){\fred\>=\ol{\bfR}_n^p(\wh{S}_n,S^*)}$ for $\wh{S}_n$ defined by
\beas 
\ol{\bfR}_n^p
&=&
n^{-1}E\bigg[\sum_{j\in\bbI_n}U(\ol{X}_\tjm,\wh{S}_n)\bigg], 
\eeas
equivalently, 
\bea\label{202608011537}
\ol{\bfR}_n^p
&=&
 \frac{1}{2n} E\bigg[\sum_{j\in\bbI_n}\bigg\{(\wh{S}_n(\ol{X}_\tjm)^{-1}-S^*(\ol{X}_\tjm)^{-1})[S^*(\ol{X}_\tjm)]+\log\frac{\det \wh{S}_n(\ol{X}_\tjm)}{\det S^*(\ol{X}_\tjm)}\bigg\}
\onecalw(\ol{X}_\tjm)\bigg]. 
\nn\\&&
\eea
The predictable generalization error is always nonnegative. 

}

{\colorro
\begin{remark}\rm
We will derive a bound for $\ol{\bfR}_n^p$. 
Obviously, the same bound remains valid for the risk when $\calw$ is replaced by any measurable non-negative function $\calw_0$ 
satisfying $\calw_0\leq\calw$ (e.g. the indicator function of a window).
\end{remark}
}

\begin{remark}\rm
If $X$ is stationary, then the predictable generalization error $\ol{\bfR}_n^p$ of the estimator $\wh{S}_n$ is expressed as 
$
\ol{\bfR}_n^p
=
E\big[U (\ol{X}_0;\wh{S}_n)\big]
$. 
{\fred 
By the compatibility property (\ref{202505050127}), $\ol{\bfR}_n^p$ is compatible with the ordinary $L^2$-risk 
of the estimator with respect to the stationary law of $X_0$.
}
%
\end{remark}

}

\subsection{Results}

We assume the following conditions as well as $[A1]$. 
%
\bd
\im[{\bf[A2]}] 
\bd
\im[(i)]
$\calx$ is {\cred a compact set in $\bbR^\sfd$ satisfying $\ol{\calx^o}=\calx$,} 
and $\check{\calx}$ is an open neighborhood of $\calx$. 

\im[(ii)]  The distribution of $X_t$ has a density $p^{X_t}$ on $\check{\calx}$ {\colorro for $t\in{\bbR_+}$}. 
\im[(iii)]  $\inf_{t\in{\bbR_+}}\inf_{a\in{\cred\check{\calx}}}p^{X_t}(a)>0$. 
\im[(iv)]  $\sup_{t\in{\bbR_+}}\sup_{x\in\check{\calx}}p^{X_t}(x)<\infty$. 
\im[(v)]  The functions $\{p^{X_t}\}_{t\in{\bbR_+}}$ are equicontinuous on $\check{\calx}$. 
\ed
\ed
For diffusion processes, for example, these conditions can be verified with the Malliavin calculus and support theorems. 
\halflineskip

Denote by $\{S_k\}_{k=1,...,\caln_n}$ a set of functions $S_k\in\bbS$ 
such that 
\beas
{\mathfrak F}_n
&=&
\bigcup_{k=1}^{\caln_n}
\bigg\{S\in{\mathfrak F}_n;\>\sup_{x\in\text{supp }\calw}{\cblue\big\|}S(x)-S_k(x){\cblue\big\|}<\delta_n\bigg\}
\eeas
for $\delta_n>0$. 
{\cblue Here $\|S\|=\max_{i,j}|S_{i,j}|$ for a matrix $S=(S_{i,j})$. }
%
The number $\caln_n$ is the covering number of ${\mathfrak F}_n$ by the $\delta_n$-balls with respect to the sup-norm. 
{\cred We may assume that $\caln_n\geq2$; double-count a single ball when it covers the entire family.}

{\colorb Write 
\bea\label{202606291716}
\lambda_n=\{n^{-1}(\log n)\log\caln_n\}^{1/2}
\eea}
and let 
\beas 
\gamma_n
\yeq
\lambda_n^2
\yeq
n^{-1}{\colorb(\log n)}\log\caln_n.
\eeas

\bd
\im[{\bf[A3]}] {\colorro The following balance conditions are satisfied: }
\bea&&
\limsup_{n\to\infty}\frac{\log\log n}{\log\caln_n}<\infty,
\label{202602140148}
\\&&
{\cblue \lim_{n\to\infty}\gamma_n=0,}
\label{202602140215}
\\&&
{\colorro h = O(n^{-\ep_0})\quad\text{ for some }\ep_0>0,}
\label{202606291521}
\\&&
\lim_{n\to\infty}T\gamma_n^{\frac{{\cred\sfd}}{{\colorro2}\beta}}(\log n)^{-3} = \infty. 
\label{202604131857}
\eea

\ed
%

{\gray
}


{\cred 
Let 
\bea\label{202608081238}
\ol{U}_n(S) 
&=& 
\frac{1}{n}\sum_{j\in\bbI_n}U(S,X_\tjm).
\eea
}
Here {\cred is the} main theorem. 

\begin{theorem}\label{202602140323}
{\cred Suppose that Conditions $[A1]$-$[A3]$ are satisfied. Then,} %
there exists a constant $\Csak$ such that 
\bea\label{202602140245a}
\ol{\bfR}_n\vee\ol{\bfR}_n^p 
&\leq&
\Csak\bigg(
\Delta_n{\cred+\>\wh{\Delta}_n}+\inf_{S\in\mathfrak{F}_n}
{\cred E\big[\ol{U}_n(S) \big]}
+n^{-1}(\log n)\log\caln_n+\delta_n\log n+{\colorro h_n^2}\bigg)\quad
\eea
{\cred for all $n\in\bbN$.}
\end{theorem}

%
\if
\proof 
Theorem \ref{202605140943}
is obtained 
from Theorem \ref{202602140323} 
with the aid of Lemma  \ref{202605140839}. 
\qed\halflineskip
\halflineskip
\fi

{\colorro
\begin{remark}\rm 
{\fred(i)} %
Denote by the subscript $\log$ (like $\sim_{\log}$) the relation modulo a (positive or negative) power of $\log n$. 
For example, in the case where 
\bea\label{202604131933}
{\fred n^{-1}}\log\caln_n &\sim_{\log}& \gamma_n\sim_{\log} n^{-\frac{2\beta}{2\beta+\sfd}},
\eea
then Condition (\ref{202604131857}) is rephrased as 
\bea\label{202604131938}
h n^{\frac{2\beta}{2\beta+\sfd}} &\to_{\log}& \infty. 
\eea
Then $h$ must satisfy $h\gg_{{\cred\log}}  n^{-\frac{2\beta}{2\beta+\sfd}} $.

\noindent{\fred (ii)} 
On the other hand, for the error bound provided by Theorems \ref{202602140323} 
to make sense in comparison with the discretization error, 
it should hold that $h^2\leq\gamma_n$ as well as $h\gg_{{\cred\log}}  n^{-\frac{2\beta}{2\beta+\sfd}} $. 
Thus, 
it is reasonable to consider the case where 
\bea\label{202605151652}
n^{-\frac{2\beta}{2\beta+\sfd}}\ll_{\log}\ h\ \ll_{\log}\ n^{-\frac{\beta}{2\beta+\sfd}}. 
\eea
{\cred 
If making $\beta\up\infty$ just formally in (\ref{202605151652}), we would have the conditions 
\bea\label{2026071000146}
T=nh\to\infty\ \text{ and }\ nh^2\to0. 
\eea
The conditions in (\ref{2026071000146}) are natural by analogy to the parametric case since 
the second one is the basic balance condition between $n$ and $h$. 
}

{\fred
\noindent(iii) The best rate attained by the bound (\ref{202602140245a}) is typically given by the balance 
between the terms 
$\inf_{S\in\mathfrak{F}_n}E[\ol{U}_n(S)]$ and $n^{-1}(\log n)\log\caln_n$. 
}

\end{remark}
}
\halflineskip

{\cred 
Now we weaken Condition $[A3]$ as follows. 
\bd
\im[{\bf[A3$^\flat$]}] 
Properties (\ref{202602140148}), (\ref{202602140215}) and (\ref{202606291521}) hold. 
{\cblue Moreover $T^{-1}(\log n)^3\to0$ as $n\to\infty$. }
\ed
%

Condition $[A3^\flat]$ does not assume (iv) of $[A3]$. 
In this case, naturally, the bound of the generalization error can be worsened. 

\begin{proposition}\label{202607131357}
Suppose that Conditions $[A1]$, $[A2]$ and $[A3^\flat]$ are satisfied. 
Then, 
for any sequence $\call=(\call_n)_{n\in\bbN}$ of positive numbers satisfying $\lim_{n\to\infty}\call_n=\infty$, 
there exists a constant $\Csbc$ such that 
\bea\label{202602140245a_flat}
\ol{\bfR}_n\vee\ol{\bfR}_n^p 
&\leq&
\Csbc\bigg(
\Delta_n+\inf_{S\in\mathfrak{F}_n}
{\cred E\big[\ol{U}_n(S) \big]}
+T^{-2\beta/\sfd}(\log n)^{6\beta/\sfd}\call_n
\nn\\&&\hspace{40pt}
+n^{-1}(\log n)\log\caln_n+\delta_n\log n+h_n^2\bigg)
\eea
for {\fred all} $n\in\bbN$. 
\end{proposition}

\begin{remark}\rm
Although we do not pursue this direction further here, it is straightforward to mildly reformulate the assumptions to obtain similar results to Theorem \ref{202602140323} and Proposition \ref{202607131357} under the repeated measurements setting, as follows.

Consider independent processes $(X^i_t,Y^i_t)_{t\in[0,T_0]}$ ($i\in\bbN$) 
each of which satisfies 
\beas 
dY^i_t &=& b(X^i_t)dt+\sigma(X^i_t)dw^i_t\quad(t\in[0,T_0])
\eeas
and 
\beas 
dX^i_t &=& {\cred b^{X^i}_t}dt+{\cred\sigma^{X^i}_t}dw_t^{X^i}\quad(t\in[0,T_0])
\eeas
where $(w^i,w^{X^i})$ ($i\in\bbN$) are independent pairs of Wiener processes 
($w^i$ and $w^{X^i}$ may be correlated), and the coefficients of the these stochastic differential equations 
are assumed to satisfy suitable conditions. 
Suppose that we observe $\big\{X^i_{jh},Y^i_{jh};\>i\in[N],j\in[M]\big\}$, $h=T_0/M$, 
for $N,M\in\bbN$ and a fixed $T_0>0$. 
We define the stochastic processes $X=(X_t)_{t\in\bbR_+}$ and $Y=(Y_t)_{t\in\bbR_+}$ as 
\beas 
X_t \yeq \sum_{i=1}^\infty1_{[(i-1)T_0,iT_0)}(t)X^i_{t-(i-1)T_0}
&\text{and}&
Y_t \yeq \sum_{i=1}^\infty1_{[(i-1)T_0,iT_0)}(t)Y^i_{t-(i-1)T_0}, 
\eeas
respectively, for $t\in\bbR_+$. 
Obviously, $X$ satisfies a geometric strong mixing condition. It is possible to make 
the stochastic differential equations (\ref{202509191103}) and (\ref{202608011934}) but 
on $t\in\cup_{i\in\bbN}[(i-1)T_0,iT_0)$. 

Let $n=MN$ and $Z=(X,Y)$ as before. The timestamps are renamed $(\tj)_{j=0,...,n}$ 
and 
we modify $\Phi_n(S,Z)$ of (\ref{202505042250}) as 
\beas
\Phi_n (S,Z) &=& \frac{1}{2n}\sum_{j\in\bbI_n}\big\{h^{-1}S(X_\tjm)^{-1}[(Y_{\tj-}-Y_\tjm)^{\otimes2}]+\log\det S(X_\tjm)\big\}\onecalw(X_\tjm). 
\eeas
%
%
Then we obtain 
similar bounds for the generalization errors under the repeated measurements 
when $N,M\to\infty$ since the proof in Section \ref{202607151444} is still valid under suitable assumptions. 
For example, Condition $[A1]$ (iv) should be replaced by a condition like
\beas 
\sup_{(N,M)\in\bbN^2}\sup_{j\in\bbI_n}\big\|h^{-1/2}|Y_{\tj-}-Y_\tjm|\big\|_{p_0} &<& \infty.
\eeas
Note that the variables $Y_{\tj-}$ are observable. 
\end{remark}
}

\section{{\cred Proof of Theorem \ref{202602140323} and Proposition \ref{202607131357}}}\label{202607151444}
Conditions $[A1]$, {\cred$[A2]$ and $[A3^\flat]$} are assumed {\cred throughout, unless otherwise stated}. 
{\cred Let $n\in\bbZ_{\geq2}$ in what follows. }%
We begin with the strategy to prove {\cred Theorem \ref{202602140323}}. 

\subsection{Strategy}\label{202606261851}
The predictable empirical risk of $\wh{S}_n$ 
is {\cred defined as} 
\bea\label{202602121114}
\cale_n^p
&=&
n^{-1}\sum_{j\in\bbI_n}U(X_\tjm,\wh{S}_n)
\\&=&
\frac{1}{2n} \sum_{j\in\bbI_n}\bigg\{(\wh{S}_n(X_\tjm)^{-1}-S^*(X_\tjm)^{-1})
 \big[S^*(X_\tjm)\big]+\log\frac{\det \wh{S}_n(X_\tjm)}{\det S^*(X_\tjm)}\bigg\}
\onecalw(X_\tjm).
\nn
\eea
%
%
The empirical risk of $\wh{S}_n^\DL$ 
is define as
\bea\label{202602120134a}
\cale_n^\DL
&=&
\Phi_n(\wh{S}_n^\DL,Z)-\Phi_n(S^*,Z)
\\&=&
\frac{1}{2n} \sum_{j\in\bbI_n}\bigg\{(\wh{S}_n^\DL(X_\tjm)^{-1}-S^*(X_\tjm)^{-1})
 \big[h^{-1}(\Delta_jY)^{\otimes2}\big]+\log\frac{\det \wh{S}_n^\DL(X_\tjm)}{\det S^*(X_\tjm)}\bigg\}
\onecalw(X_\tjm)
\nn
\eea
%
%
Moreover, the predictable empirical risk of $\wh{S}_n^\DL$ 
is define as
\bea\label{202602120134}
\cale_n^{\DL,p}
&=&
n^{-1}\sum_{j\in\bbI_n}U(X_\tjm,\wh{S}_n^\DL)
\\&=&
\frac{1}{2n} \sum_{j\in\bbI_n}\bigg\{(\wh{S}_n^\DL(X_\tjm)^{-1}-S^*(X_\tjm)^{-1})
 \big[S^*(X_\tjm)\big]+\log\frac{\det \wh{S}_n^\DL(X_\tjm)}{\det S^*(X_\tjm)}\bigg\}
\onecalw(X_\tjm)
\nn
\eea
Let $\bfR_n^p=E[\cale_n^p]$, $\bfR_n^\DL=E[\cale_n^\DL]$ and $\bfR_n^{\DL,p}=E[\cale_n^{\DL,p}]$.

In Section \ref{202606261834}, we prove {\cred Theorem \ref{202602140323} and Propositions \ref{202607131357}}. 
{\colorro 
Our strategy for the proof is as follows. 
\begin{enumerate}[(i)]
\im 
$\bfR_n^p$ is estimated by $\bfR_n^{\DL,p}$ 
(Lemma \ref{202602120344}).

\im $\bfR_n^{\DL,p}$ is estimated by $\bfR_n^\DL$ (Lemma \ref{202602140214}). 
This is by assessing the difference $\bfR_n^{\DL,p}-\bfR_n^\DL$  
with a large deviation argument (Lemma \ref{202602130518}). 
The discretization error {\colorro$h^2$} appears here. 
\label{202605140833}

\im 
$\bfR_n^\DL$ is controlled due to optimization (Lemma \ref{202604111325}). 

\im An estimate of $\bfR_n^p$ follows from Steps (i)-(iii) 
(Lemma \ref{202602140253}). 

\im 
$\ol{\bfR}_n^p$ is estimated by $\bfR_n^p$ (Lemma \ref{202509220400}) 
with the comparison theorem (Theorem \ref{202606111044}). 
{\cred For $\ol{\bfR}_n^p$, the estimate (\ref{202602140245a}) of Theorem \ref{202602140323} 
and the estimate (\ref{202602140245a_flat}) of Proposition \ref{202607131357} 
are obtained. }

\im $\ol{\bfR}_n$ is estimated by $\ol{\bfR}_n^p$ (Lemma \ref{202605140839}). 
This is by assessing the difference $\ol{\bfR}_n-\ol{\bfR}_n^p$ 
(the discretization error in a function of {\colorro$h^2$} once again appears here), 
with a large deviation argument essentially the same as Lemma \ref{202602130518}. 
This step is essentially the same as Step (\ref{202605140833}), except for the direction of the estimate is opposite, i.e., 
the optional risk is estimated by the predictable risk. 
\label{202602120401}
{\cred The estimate (\ref{202602140245a}) of Theorem \ref{202602140323} 
and the estimate (\ref{202602140245a_flat}) of Proposition \ref{202607131357} 
are obtained for $\ol{\bfR}_n$. }

\end{enumerate}

{\colorr 

\subsection{Domination of $\bfR_n^p$ by $\bfR_n^{\DL,p}$}
\begin{lemma}\label{202602120344}
There exists a constant $\Csaj$ such that 
\bea\label{202602100856}
\bfR_n^p
&\leq&
\Csaj\big(\bfR_n^{\DL,p}{\colorb +\wh{\Delta}_n}\big)
\eea
for all $n\in\bbN$. 
\end{lemma}
\proof
By using the compatibility discussed in Section \ref{202604101629}, 
we have 
\beas
\bfR_n^p
&\simleq&
E\big[\big|\wh{S}_n-S^*\big|_n^2\big]
\nn\\&\leq&
2E\big[\big|\wh{S}_n-\wh{S}^\DL_n\big|_n^2\big]+2E\big[\big|\wh{S}^\DL_n-S^*\big|_n^2\big]
\nn\\&\leq^{\fred{(\ref{202606251832})}}&
2E\big[\big|S^*-\wh{S}^\DL_n\big|_n^2\big]{\colorb +{\colorro2}\wh{\Delta}_n}+2E\big[\big|\wh{S}^\DL_n-S^*\big|_n^2\big]
\nn\\&{\fred=}&
4E\big[\big|\wh{S}^\DL_n-S^*\big|_n^2\big]{\colorb +{\colorro2}\wh{\Delta}_n}
\nn\\&\simleq&
\bfR_n^{\DL,p}{\colorb +\wh{\Delta}_n}.
\eeas
\qed

\subsection{Estimation of $\bfR_n^{\DL,p}$ by $\bfR_n^\DL$}

We are going into Step (\ref{202605140833}) of the Strategy mentioned in Section \ref{202606261851}. 

%
{\cred 
Write 
\bea\label{202606291544}
(L^Xf)_t &=& \partial f(X_t)[{\cred b_t^X}]+\half\partial^2f(X_t)[({\cred\sigma^X_t})^{\otimes2}]
\eea
for functions $f$ of class $C^2$. 
}

{\cred 
Let $\check{\calx}_0$ be an open set in $\bbR^\sfd$ such that $\calx\subset\check{\calx}_0\subset\ol{\check{\calx}_0}\subset\check{\calx}$. }%
For a stopping time $\tau$ and a set $A\in\calf_\tau$, 
we define $\tau_A$ as $\tau_A=\tau1_A+\infty1_{A^c}$. Then the random time $\tau_A$ is a stopping time. 
}
{\colorro
Thus, the random time
\beas 
\varpi_j\yeq\varpi^n_j
&=& 
\big(\tj\wedge\inf\big\{t\geq\tjm;\>X_t\not\in{\cred\check{\calx}_0}\big\}\big)_{\{\calw(X_\tjm)>0\}}\wedge(\tj)_{\{\calw(X_\tjm)=0\}}
\nn\\&=&
\big(\tj\wedge\inf\big\{t\geq\tjm;\>X_t\not\in{\cred\check{\calx}_0}\big\}\big)1_{\{\calw(X_\tjm)>0\}}+\tj1_{\{\calw(X_\tjm)=0\}}
\eeas
is a stopping time for every $j\in\bbI_n$ and $n\in\bbN$. 
}

\begin{lemma}\label{202602132307}
\bd
\im[(a)] 
{\colorro
It holds that 
\beas
\big(h^{-1}(\Delta_jY)^{\otimes2}-S^*(X_\tjm)\big)\calw(X_\tjm)1_{\{\varpi_j=\tj\}}
&=&
\big(
M^n_j(\tj)+\dot{r}^n_j
\big)\calw(X_\tjm)1_{\{\varpi_j=\tj\}}
\eeas 
}
for 
\bea\label{202602130220}
M^n_j(u)
&=&
h^{-1}\bigg\{\bigg(\int_\tjm^u\sigma^*(X_t)dw_t\bigg)^{\otimes2}-\int_\tjm^uS^*(X_t)dt\bigg\}
\nn\\&&
+2b^*(X_\tjm)\wt{\otimes}\int_\tjm^u\sigma^*(X_t)dw_t
+h^{-1}\int_\tjm^u(\tj-t)\partial S^*(X_t){\colorro [{\cred\sigma^X_t}dw_t^X]}
\eea
and a random variable $\dot{r}^n_j$ taking values in $\bbR^\sfm\otimes\bbR^\sfm$ such that 
\bea\label{202606301728}
\sup_{n\in\bbN}\sup_{j\in\bbI_n}\big(h^{-1}\big\||\dot{r}^n_j|
{\colorro1_{\{\varpi_j=\tj\}}}
\big\|_p\big) &<& \infty
\eea
for every $p>1$. Here $\wt{\otimes}$ denotes the symmetrized tensor product.
(Note that $\tj-t$ is not $u-t$ in the integrand of the last term on the right-hand side of (\ref{202602130220}). 
It is to make a martingale. 
{\colorro Also remark that $\partial S^*(X_t)$ in (\ref{202602130220}) makes sense on $\{\calw(X_\tjm)>0\}\cap\{\varpi_j=\tj\}$.}
)

\im[(b)] 
{\cred The difference} %
$\cale_n^\DL-\cale_n^{\DL,p}$ admits the equality 
\bea\label{202602120106}
{\colorro1_{\cap_{j\in\bbI_n}\{\varpi_j=\tj\}}}
\big(\cale_n^\DL-\cale_n^{\DL,p}\big)
&=&
{\colorro1_{\cap_{j\in\bbI_n}\{\varpi_j=\tj\}}}\big(
\bbJ_n
+r_n\big),
\eea
where 
\bea\label{202602120111}
\bbJ_n
\yeq
\frac{1}{2n} \sum_{j\in\bbI_n}\big(\wh{S}_n^\DL(X_\tjm)^{-1}-S^*(X_\tjm)^{-1}\big)
 \big[M^n_j(\tj{\colorro\wedge\varpi_j})\big]
\onecalw(X_\tjm)
\eea
%
{\colorro and 
\bea\label{202606301704}
r_n
&=&
\frac{1}{2n} \sum_{j\in\bbI_n}\big(\wh{S}_n^\DL(X_\tjm)^{-1}-S^*(X_\tjm)^{-1}\big)
 \big[\dot{r}^n_j1_{\{\varpi_j=\tj\}}\big]
\onecalw(X_\tjm).
\eea
}
\ed
\end{lemma}
\proof 
{\cred By $[A1]$ (iii) regarding $b^*$, }%
{\colorro on the event $\{\varpi_j=\tj\}\cap\{\calw(X_\tjm)>0\}$,}
the increment $\Delta_jY$ has the decomposition
\beas 
\Delta_jY
&=& 
\int_\tjm^\tj\sigma^*(X_t)dw_t+\int_\tjm^\tj b^*(X_t)dt
\nn\\&=&
\int_\tjm^\tj\sigma^*(X_t)dw_t+b^*(X_\tjm)h
+\wt{r}_j^n,
\eeas
where 
\beas 
\wt{r}_j^n
&=&
\int_\tjm^\tj \int_\tjm^t \partial b^*(X_s)\sigma^*(X_s)dw_sdt+\int_\tjm^\tj \int_\tjm^t L^{{\cred X}}{\cred b^*}(X_s)dsdt
\eeas
{\fred with $L^X$ defined in (\ref{202606291544}). }%
{\colorro
Extend $\wt{r}_j^n$ as $\wt{r}_j^n=0$ on 
$\big(\{\varpi_j=\tj\}\cap\{\calw(X_\tjm)>0\}\big)^c$. 
Then, {\cred thanks to $[A1]$ (ii) and (iii) (the condition regarding $b^*$),} we have the estimate
}
\beas 
\sup_{n\in\bbN}\sup_{j\in\bbI_n}\big(h^{-3/2}\big\||\wt{r}_j^n|{\colorro1_{\{\varpi_j=\tj\}}}\big\|_p\big) &<& \infty
\eeas
for every $p>1$.%
%

%
{\colorro On the event $\{\varpi_j=\tj\}\cap\{\calw(X_\tjm)>0\}$,} 
{\cred by using $[A1]$ (iii) for $S^*$, }%
we have
%
\beas 
\int_\tjm^\tj S^*(X_t)dt-hS^*(X_\tjm) 
&=& 
\int_\tjm^\tj\int_\tjm^t\partial S^*(X_s)
{\colorro [{\cred\sigma^X_s}dw_s^X]}dt+\int_\tjm^\tj\int_\tjm^t {\colorro( L^XS^*)_s}dsdt.
\eeas
%
Thus we obtain (a) {\colorro by using (\ref{202606291544}), and (\ref{202606291546})} {\cred (i.e., $[A1]$ (ii) on the event)}. 
The property (b) follows from (a). 
\qed\halflineskip

We will estimate $E[\bbJ_n]$. 
For each $j\in\bbI_n$, let 
\beas 
{\colorb\dot{\tau}_j} \yeq {\colorb\dot{\tau}_j^n} 
&=&
{\colorro\varpi_j\wedge}
\inf\bigg\{t\geq\tjm;\>
{\cred |\sigma^*(X_t)|+
|\partial S^*(X_t)|>\sup_{x\in\text{supp }\calw}\big(|\sigma^*(x)+|\partial S^*(x)|\big)+1}\bigg\}
\nn\\&&
\wedge\inf\bigg\{t\geq\tjm;\>\bigg|h^{-1/2}\int_\tjm^t\sigma^*(X_s)dw_s\bigg|\geq\Csab(\log n)^{1/2}\bigg\}
\nn\\&&
\wedge\inf\bigg\{t\geq\tjm;\>\bigg|h^{-1}\int_\tjm^t(\tj-s)\partial S^*(X_s){\colorro [{\cred\sigma^X_s}dw_s^X]}\bigg|\geq
{\colorro\log n}
\bigg\}
\nn\\&&
{\colorro
\wedge\inf\bigg\{t\geq\tjm;\>\int_\tjm^t|{\cred\sigma^X_s}|^2ds>\log n\bigg\}.
}
\eeas
%
%
{\colorb
Define $\tau_j=\tau_j^n$ as 
$
\tau_j 
=
t_j\wedge\big(\dot{\tau}_j\big)_{\{\calw(X_\tjm)\not=0\}}
$, 
in other words, 
$
\tau_j 
=
\dot{\tau}_j1_{\{\calw(X_\tjm)\not=0\}}+t_j1_{\{\calw(X_\tjm)=0\}}
$. 
The random time $\tau_j$ is a stopping time since 
$\{\calw(X_\tjm)=0\}\in\calf_\tjm\subset\calf_{\dot{\tau}_j}$. 
}%

If we take a sufficiently large constant $\Csab$, then 
\bea\label{202607031437}
{\colorro P\big[\cup_{j\in\bbI_n}\{\varpi_j<\tj\}\big] 
\yleq\> }
P\big[\cup_{j\in\bbI_n}\{\tau_j<\tj\}\big]
\yeq
{\colorro O(n^{-L})}
\eea
{\colorro as $n\to\infty$, for any positive constant $L$. 
Here we used 
{\cred (\ref{202606291521}) and $[A1]$ (ii)-(iii) regarding the continuity of $\sigma^*$ and $\partial S^*$ on $\ol{\check{\calx}_0}$. }%
}
Thus, 
\bea\label{202602130124}&&
E\big[|\bbJ_n|1_{\{\cup_{j\in\bbI_n}\{\tau_j<\tj\}}\big]
\nn\\&\leq&
\bigg(\sup_{{\cred S\in\bbS}\atop x\in\text{supp }\calw}|S(x)^{-1}{\cred\calw(x)}|\sup_{n\in\bbN}\sup_{j\in\bbI_n}\big\||M^n_j(\tj{\colorro\wedge\varpi_j})|\big\|_2\bigg)
\bigg(P\big[\cup_{j\in\bbI_n}\{\tau_j<\tj\}\big]\bigg)^{1/2}
\nn\\&\simleq&
\Csu n^{-1}\quad(n\in\bbN)
\eea
for some constant $\Csu$, 
{\cred where $[A1]$ (ii) is used to estimate $M^n_j(\tj\wedge\varpi_j)$.}
%

Let 
\bea\label{202602130215}
\bbJ_n^*(t)
&=&
\frac{1}{2n} \sum_{j\in\bbI_n} 1_{\{t\geq\tjm\}}
\big(\wh{S}_n^\DL(X_\tjm)^{-1}-S^*(X_\tjm)^{-1}\big)
 \big[M^n_j(t\wedge\tau_j)\big]
\onecalw(X_\tjm).
\eea

From (\ref{202602130220}), {\cred there exists constant $\Csac$ such that}
\bea\label{202602150412}
{\cred\|\calw\|_\infty}
\sup_{j\in\bbI_n}\sup_{t\in[\tjm,\tj]}\big\||M^n_j(t\wedge\tau_j)| {\colorro1_{\text{supp }\calw}(X_\tjm)} \big\|_\infty
&\leq&
\Csac{\fred\sfm^{2}}(\Csab\vee1)^2(\log n+1)
\eea
for all $n\in\bbN$. 

Let $\sfk$ be a $\{1,...,\caln_n\}$-valued random number such that 
\bea\label{202604101633}
\sup_{x\in\text{supp }\calw}{\cred\big\|}\wh{S}_n^\DL(x)-S_\sfk(x){\cred\big\|}<\delta_n. 
\eea
Moreover, let 
\beas
\bbJ_{n,k}^*(t)
&=&
\frac{1}{2n} \sum_{j\in\bbI_n} 1_{\{t\geq\tjm\}}
\big(S_k(X_\tjm)^{-1}-S^*(X_\tjm)^{-1}\big)
 \big[M^n_j(t\wedge\tau_j)\big]
\onecalw(X_\tjm)
\eeas
{\cred for $k\in\{1,...,\caln_n\}$.} 
Then 
\bea\label{202602130302}
\sup_{t\in\bbR_+}\big|\bbJ^*_n(t)-\bbJ^*_{n,\sfk}(t)\big|
&\leq^{(\ref{202602150412}), (\ref{202604101633})}&
\Csac
{\fred m^2}
(\Csab\vee1)^2
(\log n+1)\delta_n\qquad(n\in\bbN).
\eea

{\colorro Recall that $\lambda_n$ is defined by (\ref{202606291716}). }
\begin{lemma}\label{202602130518}
There exists a constant $\Csah$ such that 
\bea\label{202602130517}
E\big[\big|{\colorb\big(|S_\sfk-S^*|_n\vee\lambda_n\big)^{-1}}\bbJ_{n,\sfk}^{*}(nh)\big|^2\big]
&\leq&
\Csah n^{-1}(\log n)\log\caln_n
\eea
for all $n\in\bbZ_{\geq2}$. 
\end{lemma}
\proof
For every $n\in\bbN$ and $k\in\{1,...,\caln_n\}$, 
{\colorb the process} $(\bbJ^*_{n,k}(t))_{t\in\bbR+}$ 
is a continuous local martingale with respect to $\bbF$ 
{\cred whose quadratic variation process $\big\langle \bbJ^*_{n,k} \big\rangle$ satisfies 
\beas 
\big\langle \bbJ^*_{n,k} \big\rangle_t
&\leq&
\Csad n^{-2}\sum_{j\in\bbI_n}1_{\{t\geq\tjm\}}\big|S_k(X_\tjm)-S^*(X_\tjm)\big|^2\onecalw(X_\tjm)
\nn\\&&\hspace{50pt}\times
\bigg\{h^{-2}\int_\tjm^{t\wedge\tau_j}|\sigma^*(X_s)|^2\bigg|\int_\tjm^s\sigma^*(X_r)dw_r\bigg|^2ds
\nn\\&&\hspace{70pt}
+|b^*(X_\tjm)|^2\int_\tjm^{t\wedge\tau_j}|\sigma^*(X_s)|^2ds
\nn\\&&\hspace{70pt}
+h^{-2}\int_\tjm^{t\wedge\tau_j}(\tj-s)^2\big|\partial S^*(X_s)\big|^2|\sigma^X_s|^2ds\bigg\},
\eeas
where $\Csad$ is a constant independent of $(n,k,t)$. 
}
%
%
%
{\cred By the definition of $\tau_j$,} we obtain 
\bea\label{202604101642}
\big\langle \bbJ^*_{n,k} \big\rangle_{nh}
&\leq&
\Csai  n^{-2}\sum_{j\in\bbI_n}\big|S_k(X_\tjm)-S^*(X_\tjm)\big|^2\onecalw(X_\tjm)\big(\log n+h\big)
\nn\\&=&
\Csai n^{-1} |S_k-S^*|_n^2\big(\log n+h\big)
\eea
for some constant $\Csai$ {\cred for all $n\in\bbN$.}

Define integers $\nu_0$ and $\nu_n$ ($n\in\bbN$) as 
$
2^{-\nu_0+1}> 2\sup_{x\in\text{supp }\calw}\big(|S(x)||\calw(x)^{1/2}|\big)$ and 
\beas
\nu_n\yeq \inf\big\{\nu;\> 2^{-\nu}\leq {\colorb\lambda_n^3(\log n)^{-1}
}
\big\} 
\quad(n\in\bbN),
\eeas
respectively. 
Then 
\bea\label{202604101641}
\nu_n\simleq\log n. 
\eea
{\colorro 
Indeed, we have 
\beas 
\lambda_n^3(\log n)^{-1}
&=&
n^{-3/2}(\log n)^{1/2}(\log\caln_n)^{3/2}
\ygeq 
n^{-2}
\eeas
for large $n$, since it is assumed that $\caln_n\geq2$. 

}

%
%
{\colorb
Since
\beas
|\bbJ_{n,k}^*(nh)|
&\simleq&
\frac{1}{n} \sum_{j\in\bbI_n} 
\big|S_k(X_\tjm)-S^*(X_\tjm)\big|
 |M^n_j(t\wedge\tau_j)|
 {\colorro \calw(X_\tjm)}
 \nn\\&\simleq&
 {\colorro
 \bigg\{\frac{1}{n}\sum_{j\in\bbI_n}\big|S_k(X_\tjm)-S^*(X_\tjm)\big|^2\calw(X_\tjm)\bigg\}^{1/2}
  }
 \nn\\&& {\colorro\times
  \bigg\{\frac{1}{n}\sum_{j\in\bbI_n} |M^n_j(t\wedge\tau_j)|^2\calw(X_\tjm)\bigg\}^{1/2}
 }
 \nn\\&\simleq^{(\ref{202602150412})}&
 \big|S_k-S^*\big|_n\log n, 
 \eeas
 we have
 \bea\label{202602150422}
1_{\{|S_k-S^*|_n<2^{-\nu_n}\}}|\bbJ_{n,k}^*(nh)|
&\leq&
\Csao \lambda_n^3
\eea
for some constant $\Csao$. 
}

Applying the martingale representation theorem with a Brownian motion $(B(t))_{t\in\bbR_+}$ to $(\bbJ^*_{n,k}(t))_{t\in\bbR+}$, we obtain 
\bea\label{202602132146}&&
P\big[|\bbJ^*_{n,k}(nh)|>x|S_k-S^*|_n,{\>\colorb |S_k-S^*|_n\geq2^{-\nu_n}}\big]
\nn\\&\leq&
\sum_{\nu=\nu_0}^{\nu_n}P\big[|\bbJ^*_{n,k}(nh)|>x|S_k-S^*|_n,\>
|S_k-S^*|_n\in[2^{-\nu},2^{-\nu+1})\big]
\nn\\&\leq&
\sum_{\nu=\nu_0}^{\nu_n}P\big[|\bbJ^*_{n,k}(nh)|>2^{-\nu}x,\>
|S_k-S^*|_n\leq2^{-\nu+1}\big]
\nn\\&\simleq^{(\ref{202604101642})}&
\sum_{\nu=\nu_0}^{\nu_n}P\bigg[\sup\bigg\{
|B(t)|>2^{-\nu}x;\>t\in\big[0,\Csai 2^{-2\nu+2}{\colorb n^{-1}}(\log n+h)\big]\bigg\}
\bigg]
\nn\\&\simleq^{(\ref{202604101641})}&
\frac{n^{-1/2}\big(\log n\big)^{1/2}\log n}{x}
\exp\bigg(-\frac{\Csaf nx^2}{\log n }\bigg)
\nn\\&\leq&
\Csae 
\frac{n^{-1/2}\big(\log n\big)^{3/2}}{x}
\exp\bigg(-\frac{\Csaf nx^2}{\log n}\bigg)
\eea
for all $x>0$ and $n\in\bbZ_{\geq2}$, for some constant $\Csae$ independent of $n$ and $k\in\{1,...,\caln_n\}$. 
Here $\Csaf$ is a universal positive constant. 
{\colorb Consequently, 
\bea\label{202602150426}&&
P\big[|\bbJ^*_{n,k}(nh)|>x\big(|S_k-S^*|_n{\colorb\vee\lambda_n}\big)\big]
\nn\\&\leq&
P\big[|\bbJ^*_{n,k}(nh)|>x|S_k-S^*|_n,{\>\colorb |S_k-S^*|_n\geq2^{-\nu_n}}\big]
\nn\\&&
+P\big[|\bbJ^*_{n,k}(nh)|>x{\colorb\lambda_n},{\>\colorb |S_k-S^*|_n<2^{-\nu_n}}\big]
\nn\\&\leq^{(\ref{202602132146}),(\ref{202602150422})}&
\Csae 
\frac{n^{-1/2}\big(\log n\big)^{3/2}}{x}
\exp\bigg(-\frac{\Csaf nx^2}{\log n}\bigg)
+1_{\{\Csao \lambda_n^2>x\}}. 
\eea
}

For any positive constant $\Csag$ 
{\colorb satisfying 
\beas
\Csag\ygeq^{(\ref{202602140215})} \sup_{n\in\bbZ_{\geq2}}\Csao\lambda_n,
\eeas
(we are assuming $\caln_n\geq2$) 
}
it holds that 
\beas&&
E\big[\big|
{\colorb\big(|S_\sfk-S^*|_n{\colorb\vee\lambda_n}\big)^{-1}}
\bbJ_{n,\sfk}^{*}(nh)\big|^2\big]
\nn\\&=& 
\int_0^\infty 2xP\big[\big|\bbJ_{n,\sfk}^{*}(nh)\big|>x
(|S_\sfk-S^*|_n{\colorb\vee\lambda_n}\big)
\big] dx
\nn\\&\leq&
\int_0^{\Csag{\colorb\lambda_n}} 2xP\big[\big|\bbJ_{n,\sfk}^{*}(nh)\big|>x(|S_\sfk-S^*|_n{\colorb\vee\lambda_n}\big)\big] dx
\nn\\&&
+\int_{\Csag {\colorb\lambda_n}}^\infty 2x 
\caln_n \sup_{k=1,...,\caln_n}P\big[\big|\bbJ_{n,k}^{*}(nh)\big|>x(|S_{\cred k}-S^*|_n{\colorb\vee\lambda_n}\big)\big] dx.
\eeas
Therefore, 
\bea\label{202602130442}&&
E\big[\big|
{\colorb\big(|S_\sfk-S^*|_n{\colorb\vee\lambda_n}\big)^{-1}}
\bbJ_{n,\sfk}^{*}(nh)\big|^2\big]
\nn\\&\leq^{(\ref{202602150426})}&
\int_0^{\Csag {\colorb\lambda_n}} 2xdx
\nn\\&&
+2\Csae\>\caln_n \int_{\Csag {\colorb\lambda_n}}^\infty n^{-1/2}\big(\log n\big)^{3/2}
\exp\bigg(-\frac{\Csaf nx^2}{\log n}\bigg)dx
\nn\\&\simleq&
\Csag^2 n^{-1}(\log n)\log\caln_n
\nn\\&&
+2\Csae\>n^{-1}(\log n)^2(\log\caln_n)^{-1/2}\caln_n \exp\big(-\Csaf \Csag^2\log\caln_n\big)
\eea
due to the elementary estimate 
\beas
\int_B^\infty e^{-Cx^2}dx
&=&
\int_{C^{1/2}B}C^{-1/2}e^{-y^2}dy
\>\simleq\>
C^{-1}B^{-1}e^{-CB^2}
\eeas
for constants $B$ and $C$. 
We are assuming (\ref{202602140148}), therefore $\log n\simleq \caln_n^p$ for some $p>0$. 
Then, taking sufficiently large constants $\Csag$, we obtain 
the inequality (\ref{202602130517}). 
\qed\halflineskip

\begin{lemma}\label{202602140214}
There exists a constant $\Csal$ such that 
\bea\label{202602140207}
\bfR_n^{\DL,p}
&\leq&
\Csal\big[
\bfR_n^\DL
+n^{-1}(\log n)\log\caln_n+\delta_n\log n+{\colorro h_n^2}\big]
\eea
for all $n\in\bbN$. 
\end{lemma}
\proof
Use Lemma \ref{202602132307} to obtain 
\bea\label{202602140054}
\bfR_n^\DL-\bfR_n^{\DL,p}
&=& 
E[\cale_n^\DL]-E[\cale_n^{\DL,p}]
\nn\\&=^{(\ref{202607031437}),\>[A1]\text{(iv)}}&
{\colorro E[1_{\cap_{j\in\bbI_n}\{\varpi_j=\tj\}}(\cale_n^\DL-\cale_n^{\DL,p})]+O(n^{-1})}
\nn\\&=^{(\ref{202602120106})}&
{\colorro E[1_{\cap_{j\in\bbI_n}\{\varpi_j=\tj\}}(\bbJ_n+r_n)]+O(n^{-1})}
\nn\\&=^{(\ref{202607031437})}&
E[\bbJ_n]+E[{\colorro r_n}]{\colorro+O(n^{-1})}
\eea
By (\ref{202602130215}), 
$\bbJ_n=\bbJ^*_n(nh)$ on the event $\big(\cup_{j\in\bbI_n}\{\tau_j<\tj\}\big)^c=\cap_{j\in\bbI_n}\{\tau_j=\tj\}$. 
Therefore, 
\bea\label{202602140147}
\big|E[\bbJ_n]\big|
&\leq&
E\big[|\bbJ_n|1_{\cup_{j\in\bbI_n}\{\tau_j<\tj\}}\big]
+\big|E\big[|\bbJ_n|1_{\cap_{j\in\bbI_n}\{\tau_j=\tj\}}\big]\big|
\nn\\&\leq&
E\big[|\bbJ_n|1_{\cup_{j\in\bbI_n}\{\tau_j<\tj\}}\big]
+E\big[|\bbJ^*_n(nh)|\big]
\nn\\&\simleq^{(\ref{202602130124}),(\ref{202602130302})}& 
E\big[\big|\bbJ^*_{n,\sfk}(nh)\big|\big]+n^{-1}+\delta_n\log n
\nn\\&\leq& 
{\colorb
E\big[\big(|S_\sfk-S^*|_n+\lambda_n\big)\big|\big(|S_\sfk-S^*|_n\vee\lambda_n\big)^{-1}\bbJ^*_{n,\sfk}(nh)\big|\big]}
+n^{-1}+\delta_n\log n
\nn\\&\simleq^{\text{Lemma \ref{202602130518}}}&
\big\{\big(E\big[|S_\sfk-S^*|_n^2\big]\big)^{1/2}+\lambda_n\big\}\lambda_n+n^{-1}+\delta_n\log n
\nn\\&\simleq^{(\ref{202604101633})}&
\big\{\big(E\big[|\wh{S}_n^\DL-S^*|_n^2\big]\big)^{1/2}+\delta_n\big\}\lambda_n+\lambda_n^2+\delta_n\log n
\nn\\&=&
\big(E\big[|\wh{S}_n^\DL-S^*|_n^2\big]\big)^{1/2}\lambda_n
+\delta_n\lambda_n
+\lambda_n^2+\delta_n\log n
\nn\\&\simleq^{(\ref{202602140215})}&
\big(E\big[|\wh{S}_n^\DL-S^*|_n^2\big]\big)^{1/2}\lambda_n
+\lambda_n^2+\delta_n\log n.
\eea

Compatibility (\ref{202505050127}) combined with (\ref{202602140147}) yields 
\bea\label{202602140156}
\big|E[\bbJ_n]\big|
&\simleq&
{\colorb
\big(\bfR_n^{\DL,p}\big)^{1/2}\lambda_n
+\lambda_n^2+\delta_n\log n.
}
\eea

{\colorro
We estimate $E[|r_n|]$ for $r_n$ of (\ref{202606301704}) as follows. 
First, remark that $|{\tt A}[{\tt B}]|\leq|{\tt A}||{\tt B}|$ for matrices ${\tt A}$ and ${\tt B}$ of the same size, and that 
there exists a constant $\Csax$ such that 
$|S(x)^{-1}-S^*(x)^{-1}|\leq\Csax|S(x)-S^*(x)|$ for all $x\in\calx$ and $S\in\bbS$. 
Then 
\bea\label{202605151523}
E[|r_n|]
&\leq&
E\bigg[\frac{1}{2n} \sum_{j\in\bbI_n}\big|\wh{S}_n^\DL(X_\tjm)^{-1}-S^*(X_\tjm)^{-1}\big|
 \big|\dot{r}^n_j1_{\{\varpi_j=\tj\}}\big|
\onecalw(X_\tjm)\bigg]
\nn\\&\simleq&
\big(E\big[|\wh{S}_n^\DL-S^*|_n^2\big]\big)^{1/2}
\bigg(E\bigg[n^{-1}\sum_{j\in\bbI_n}|\dot{r}^n_j1_{\{\varpi_j=\tj\}}|^2\bigg]\bigg)^{1/2}
\nn\\&\simleq^{(\ref{202505050127}),(\ref{202606301728})}&
\big(\bfR_n^{\DL,p}\big)^{1/2}h.
\eea
}

From (\ref{202602140054}), (\ref{202602140156}) {\colorro and (\ref{202605151523})}, we obtain 
\bea\label{202602140206}
\bfR_n^{\DL,p}-\bfR_n^\DL
&\leq&
\big|E[\bbJ_n]\big|+{\colorro E[|r_n|]}{\cred\>+\>O(n^{-1})}
\nn\\&\simleq&
{\colorb
\big(\bfR_n^{\DL,p}\big)^{1/2}{\colorro(}\lambda_n{\colorro+h)}
+\lambda_n^2+\delta_n\log n
}
\eea
{\colorro since $n^{-1}\ll\lambda_n^2$ for large $n$. }%
Solve the inequality (\ref{202602140206}) in $\big(\bfR_n^{\DL,p}\big)^{1/2}$ to obtain 
\beas
\bfR_n^{\DL,p}
&\simleq&
\bfR_n^\DL
+n^{-1}(\log n)\log\caln_n+\delta_n\log n+{\colorro h^2},
\eeas
which completes the proof. 
\qed\halflineskip

\subsection{Estimation of $\bfR_n^\DL$}
\begin{lemma}\label{202604111325}
{\colorb There exists a constant $\Csar$ independent of $n\in\bbN$ such that }
\bea\label{202602140240a}
\bfR_n^\DL
&\leq&
\Delta_n
{\colorro +\inf_{S\in\mathfrak{F}_n}E\big[\Phi_n(S,Z)-\Phi_n(S^*,Z)\big]}.
\eea
\end{lemma}
\proof
The expected empirical error $\bfR_n^\DL$ is estimated by 
\beas 
\bfR_n^\DL
&=&
E[\cale_n^\DL]
\nn\\&=&
E\big[\Phi_n(\wh{S}_n^\DL,Z)-\Phi_n(S^*,Z)\big]
\nn\\&=&
E\big[\Phi_n(\wh{S}_n^\DL,Z)-\Phi_n(S,Z)\big]+E\big[\Phi_n(S,Z)-\Phi_n(S^*,Z)\big]
\nn\\&\leq&
\Delta_n+E\big[\Phi_n(S,Z)-\Phi_n(S^*,Z)\big]
\eeas
for any $S\in\mathfrak{F}_n$. 
Consequently, 
\beas\label{202602140240}
\bfR_n^\DL
&\leq&
\Delta_n+\inf_{S\in\mathfrak{F}_n}E\big[\Phi_n(S,Z)-\Phi_n(S^*,Z)\big].
\eeas
%
%
Therefore, the inequality (\ref{202602140240a}) has been obtained. 
\qed\halflineskip

\subsection{Estimate of $\bfR_n^p$}
\begin{lemma}\label{202602140253}
There exists a constant $\Csbh$ such that 
\bea\label{202602140245}
\bfR_n^p
&\leq&
\Csbh\bigg(
\Delta_n
{\colorb +\wh{\Delta}_n}
{\colorro +\inf_{S\in\mathfrak{F}_n}E\big[\Phi_n(S,Z)-\Phi_n(S^*,Z)\big]}
\nn\\&&\hspace{50pt}
+n^{-1}(\log n)\log\caln_n+\delta_n\log n+{\colorro h_n^2}\bigg)
\eea
for all $n\in\bbN$. 
\end{lemma}
\proof
We obtain (\ref{202602140245}) by Lemmas \ref{202602120344}, \ref{202602140214} and \ref{202604111325}. 
\qed\halflineskip

\subsection{$\ol{\bfR}_n^p$ estimated by $\bfR_n^p$}

%

\begin{lemma}\label{202509220400}
{\cred Assume $[A1]$ $(i)$ and $[A2]$. 
}
\bd
\im[(a)]
{\cred 
Let $(\wt{\gamma}_n)_{n\in\bbN}$ be a sequence of positive numbers satisfying $\lim_{n\to\infty}\wt{\gamma}_n=0$. }%
{\cred 
Moreover, assume that 
\bea
\lim_{n\to\infty}T\>\wt{\gamma}_n^{\>\frac{{\fred\sfd}}{{\colorro2}\beta}}(\log n)^{-3} = \infty. 
\label{202604131857a}
\eea
}
{\cred Then} there exists a constant $\Csk$ 
independent of $n$,  such that 
\bea\label{202509240317}
\ol{\bfR}_n^p &\leq& \Csk\big(\bfR_n^p+{\cred\wt{\gamma}_n}\big)
\eea
for all $n\in\bbN$. 

\im[(b)] 
{\cblue Suppose that $T^{-1}(\log n)^3\to0$ as $n\to\infty$. Then, }%
{\cred 
for any sequence $\call=(\call_n)_{n\in\bbN}$ of positive numbers satisfying $\lim_{n\to\infty}\call_n=\infty$, 
there exists a constant $\Csay$ depending on $\ep$ and $\call$ but independent of $n$ such that 
\bea\label{202509240317b}
\ol{\bfR}_n^p &\leq& \Csay\big(\bfR_n^p
+T^{-2\beta/\sfd}(\log n)^{6\beta/\sfd}\call_n
\big)
\eea
for all $n\in\bbN$. 
}

\im[(c)] {\cred If {\cblue(\ref{202602140215})} and (\ref{202604131857}) hold, then 
there exists a constant $\Csaz$ 
independent of $n$ such that 
\bea\label{202509240317c}
\ol{\bfR}_n^p &\leq& \Csaz\big(\bfR_n^p+\gamma_n\big)
\eea
for all $n\in\bbN$. 

}

\ed
\end{lemma}
\proof 
Since ${\cred\wt{\gamma}_n}>0$, it is sufficient to show the inequality (\ref{202509240317}) for sufficiently large $n$. 
%
{\cred 
Take an open set $\hat{\calx}$ in $\bbR^\sfd$ such that 
$\calx\subset\hat{\calx}\subset\ol{\hat{\calx}}\subset\check{\calx}$. 
}
{\colorro 
We apply the inequality (\ref{202606121049}) of Theorem \ref{202606111044} for 
{\cred $\ep>0$, $d_n=\wt{\gamma}_n^{\>\frac{1}{{\colorro 2}\beta}}$, 
$\psi=1_{\check{\calx}}$, 
$m=1$, 
{\cred 
$D_n(x)=(D_n^{(i)}(x))_{i\in[\sfd_D]}=\big(\wh{S}_n(x)-S^*(x)\big)\calw(x)^{1/2}$ 
with $\sfd_D=\sfm^2$, 
}%
$\xi'_{n,j}=\ol{X}_\tjm$ and  $\xi_{n,j}=X_\tjm$. 
}
{\colorb 
We can regard each ${\cred D_n^{(i)}}\in C_b^{\ell,\beta-\ell}(\bbR^\sfd)$ 
with the support in $\calx$, due to Lemma \ref{202604121758}. 
}
{\cred 
Accordingly, we define the set ${\tt B}$ as 
\beas 
{\tt B}
&=& 
\big\{\big((S-S^*)\calw^{1/2}\big)_{i,j};\>S\in\sfB,\>(i,j)\in[\sfm]^2\big\},
\eeas
where 
$\big((S-S^*)\calw^{1/2}\big)_{i,j}$ is the $(i,j)$-components of the matrix-valued function $(S-S^*)\calw^{1/2}$, and 
each component $\big((S-S^*)\calw^{1/2}\big)_{i,j}$ is regarded as 
the $zero$-extension of $\big((S-S^*)\calw^{1/2}\big)_{i,j}|_{\calx^o}$ from $\calx^o$ to $\bbR^\sfd$. 
Then ${\tt B}$ is a bounded set in $C_b^{\ell,\beta-\ell}(\bbR^\sfd)$. 
}
{\colorro
Thus, we obtain 
\bea\label{202607011407}
\bbR_n\big((\ol{X}_\tjm)_{j\in\bbI_n}\big) 
&\leq& 
\Csbi
\bigg[
\bbR_n\big((X_\tjm)_{j\in\bbI_n}\big) 
+d_n^{2\beta}
+d_n^{-\sfd}
\sup_{a\in{\cred\hat{\calx}}}Q_{1,n}(a)
\bigg]
\eea
for all $n\in\bbN$, 
where $\Csbi$ is a constant, 
\beas 
\bbR_n\big((\ol{X}_\tjm)_{j\in\bbI_n}\big) 
&=& 
E\bigg[n^{-1}\sum_{j\in\bbI_n}\big|\wh{S}_n(\ol{X}_\tjm)-S^*((\ol{X}_\tjm)\big|^2\calw(\ol{X}_\tjm)\bigg],
\eeas
\beas 
\bbR_n\big((X_\tjm)_{j\in\bbI_n}\big) 
&=& 
E\bigg[n^{-1}\sum_{j\in\bbI_n}\big|\wh{S}_n(X_\tjm)-S^*((X_\tjm)\big|^2\calw(X_\tjm)\bigg]. 
\eeas
and 
\bea\label{202607011411}
Q_{1,n}(a) 
&=& 
P\bigg[
\bigg|n^{-1}\sum_{j=1}^n\bigg\{\wt{\varphi}_{d_n}(X_\tjm-a){\colorb \psi}(X_\tjm)
-E\big[\wt{\varphi}_{d_n}(X_\tjm-a){\colorb \psi}(X_\tjm)\big]\bigg\}\bigg|
\nn\\&&\hspace{60pt}
{\cred\geq\>}\frac{\ep}{2(1+\ep)}n^{-1}\sum_{j=1}^nE\big[\wt{\varphi}_{d_n}(X_\tjm-a){\colorb \psi}(X_\tjm)\big]
\bigg]
\eea 
for $\wt{\varphi}_{d_n}=\varphi_{0,d_n/4}$, 
in this situation. 
}

{\colorro
Define $\wt{{\tt K}}$ by $\wt{{\tt K}}(x)={\tt K}_0(4x)$ for $x\in\bbR$. 
}
We have 
\bea\label{202510111949}&&
\sup_{a\in{\cred\hat{\calx}}}
\bigg|E\big[\wt{\varphi}_{d_n}(X_t-a){\colorb \psi}(X_t)\big]-p^{X_t}(a)\bigg|
\nn\\&\leq&
\sup_{a\in{\cred\hat{\calx}}}\bigg\{
\int_{\check{\calx}}\big|\wt{\varphi}_{d_n}(x-a)\big|\>\big|p^{X_t}(x)-p^{X_t}(a)\big|dx
+p^{X_t}(a)\int_{(\check{\calx})^c}\big|\wt{\varphi}_{d_n}(x-a)\big|dx\bigg\}
\nn\\&\leq&
\sup_{a\in{\cred\hat{\calx}},x\in\check{\calx}:|x-a|<d_n^{1/2}}\big|p^{X_t}(x)-p^{X_t}(a)\big|
\int_{\bbR^\sfd}\prod_{i=1}^\sfd|\wt{{\tt K}}(y_i)|dx
\nn\\&&
+\sup_{x\in\check{\calx}}p^{X_t}(x)
\int_{y=(y_i):|y|\geq d_n^{-1/2}}\prod_{i=1}^\sfd|\wt{{\tt K}}(y_i)|dy
\nn\\&&
+\sup_{a\in{\cred\hat{\calx}}}
p^{X_t}(a)\int_{y=(y_i):|y|\geq d_n^{-1}\text{dist}({\cred\hat{\calx}},(\check{\calx})^c)}
\prod_{i=1}^\sfd|\wt{{\tt K}}(y_i)|dy.
\eea
Therefore, by (\ref{202510111949}), there exists $n_0$ such that 
\bea\label{202510112054}&&
2^{-1}\inf_{a\in{\cred\hat{\calx}}}p^{X_t}(a)-\inf_{a\in{\cred\hat{\calx}}}E\big[\wt{\varphi}_{d_n}(X_t-a){\colorb \psi}(X_t)\big]
\nn\\&\leq&
2^{-1}\inf_{a\in{\cred\hat{\calx}}}p^{X_t}(a)
+\sup_{a\in{\cred\hat{\calx}}}\bigg|E\big[\wt{\varphi}_{d_n}(X_t-a){\colorb \psi}(X_t)\big]-p^{X_t}(a)\bigg|
-\inf_{a\in{\cred\hat{\calx}}}p^{X_t}(a)
\nn\\&{\cred=}&
-2^{-1}\inf_{a\in{\cred\hat{\calx}}}p^{X_t}(a)
+\sup_{a\in{\cred\hat{\calx}}}\bigg|E\big[\wt{\varphi}_{d_n}(X_t-a){\colorb \psi}(X_t)\big]-p^{X_t}(a)\bigg|
\nn\\&\leq&
0
\eea
for all $n\geq n_0$, where $n_0$ does not depend on $t\in\bbR_+$.

Inequality (\ref{202510112054}) gives 
\beas
\inf_{a\in{\cred\hat{\calx}}}E\big[\wt{\varphi}_{d_n}(X_t-a){\colorb \psi}(X_t)\big]
&\geq&
2^{-1}\inf_{a\in{\cred\hat{\calx}}}p^{X_t}(a)
{\cred \>\ygeq 2^{-1}\inf_{a\in{\cred\check{\calx}}}p^{X_t}(a)}
\eeas
for all $n\geq n_0$ and $t\in\bbR_+$.

{\colorro Fix $a\in{\cred\hat{\calx}}$. }%
Now, 
{\colorro from (\ref{202607011411}),}
we have 
\bea\label{202510112247}
{\colorro Q_{1,n}(a) }
&\leq&
P\bigg[
\bigg|n^{-1}\sum_{j=1}^n
V^n_j
\bigg|
>\Css
\bigg]
\eea
for
\beas
\Css &=&
\frac{\ep}{4(1+\ep)}\inf_{a'\in{\cred\check{\calx}}}p^{X_\tjm}(a')\>>\>0
\eeas
and
\beas 
V^n_j
&=&
\wt{\varphi}_{d_n}(X_\tjm-a){\colorb \psi}(X_\tjm)-E\big[\wt{\varphi}_{d_n}(X_\tjm-a){\colorb \psi}(X_\tjm)\big]. 
\eeas
{\colorro The random variable $V^n_j$ depends on $a{\cred\>\in\hat{\calx}}$. }
{\colorb 
Let $\bbI_{n,r}=\{j\in\bbI_n;\>j=r\>(\text{mod }{\cred\lfloor1/h\rfloor})\}$ for $r\in\{0,1,....,\lfloor h^{-1}\rfloor{\cred-1}\}$. 
By (\ref{202510112247}), we estimate ${\colorro Q_{1,n}(a) }$ as 
\bea\label{202607011605}
{\colorro Q_{1,n}(a) }
&\simleq&
h^{-1}\sup_{r}P\bigg[\bigg|T^{-1}\sum_{j\in\bbI_{n,r}}V^n_j\bigg|>\Csas\bigg]
\eea
for some positive constant $\Csas$. 

{\fred The sequence $\big(V^n_j;j\in\bbN,j=r\>(\text{mod}\lfloor 1/h\rfloor)\big)$ is geometrically $\alpha$-mixing 
uniformly in $r$. 
}
We will apply Theorem 2 of Merlev\`ede et al. \cite{merlevede2009bernstein}. 
Let ${\fred\varep}$ be an arbitrary positive {\fred number}. 
For 
\beas 
v^2 &=& \sup_{j\in\bbN}\bigg(\text{Var}[V^n_j]+2\sum_{k\in\bbN,k>j}\big|\text{Cov}[V^n_j,V^n_k]\big|\bigg), 
\eeas
we have 
\bea\label{202604131829}
v^2 
&\leq& 
V(\gamma,{\fred\varep})\sup_j\big(\|V^n_j\|_{2+2{\fred\varep}}\big)^2
\eea
with
\beas 
V(\gamma,{\fred\varep})
&\leq&
4\bigg[1+4\sum_{j\in\bbN}\gamma^{-\frac{1}{1+{\fred\varep}^{-1}}}\exp\bigg(-\frac{\gamma}{1+{\fred\varep}^{-1}}j\bigg)\bigg]
\nn\\&=&
4+\frac{16\gamma^{-\frac{1}{1+{\fred\varep}^{-1}}}}{1-\exp\big(-\frac{\gamma}{1+{\fred\varep}^{-1}}\big)}.
\eeas
{\cred Here $v^2$ depends on $(n,a)$, and 
$\gamma$ is the parameter in the bound of $\alpha^X(r)$ in $[A1]$ (i).}
Inequality (\ref{202604131829}) follows from the covariance inequality for mixing processes 
{\cred (cf. Rio \cite{rio2017asymptotic} p. 6)}. 

Moreover, we have 
\beas 
\big(\|V^n_j\|_{2+2{\fred\varep}}\big)^{2+2{\fred\varep}}
&\simleq&
d_n^{-\sfd(1+2{\fred\varep})}
\yeq 
{\cred\wt{\gamma}_n}^{-\frac{{\fred\sfd}}{{\colorro2}\beta}(1+2{\fred\varep})}
\eeas
since $d_n^{{\colorro2}\beta}={\cred\wt{\gamma}_n}$. 
If we take ${\fred\varep}=\frac{1}{\log n}$, then (\ref{202604131829}) yields 
\beas 
v^2 
\>\simleq\>
{\cred\wt{\gamma}_n}^{-\frac{{\fred\sfd}}{{\colorro2}\beta}\frac{1+2{\fred\varep}}{1+{\fred\varep}}}\>\log n
\>\simleq\>
{\cred\wt{\gamma}_n}^{-\frac{{\fred\sfd}}{{\colorro2}\beta}}\log n.
\eeas
{\cred In the last inequality, the following estimate was used: for large $n$, 
\beas 
\wt{\gamma}_n^{-\frac{{\fred\sfd}}{{\colorro2}\beta}\frac{{\fred\varep}}{1+{\fred\varep}}}
&\leq&
\big\{T(\log n)^{-3}\big\}^{\frac{{\fred\varep}}{1+{\fred\varep}}}
\quad(\text{by }(\ref{202604131857a}))
\nn\\&\leq&
\big\{n(\log n)^{-3}\big\}^{\frac{{\fred\varep}}{1+{\fred\varep}}}
\quad(h\to0)
\nn\\&\leq&
\big\{n(\log n)^{-3}\big\}^{{\fred\varep}}
\nn\\&\leq& 
2e.
\eeas

}

We consider the coefficient appearing in Inequality (8) of Theorem 2 of Merlev\`ede et al. \cite{merlevede2009bernstein}: 
\beas 
\frac{x^2}{v^2n+M^2+xM(\log n)^2}
\eeas
in their notation. 
In the present case, we set 
\beas 
n\leftarrow {\colorro T}, \quad
x\leftarrow T, \quad M\leftarrow \sup_j\|V^n_j\|_\infty\simleq d_n^{-\sfd}={\cred\wt{\gamma}_n}^{-\frac{\sfd}{{\colorro2}\beta}}
\eeas
to obtain 
\beas 
\frac{T^2}{v^2T+M^2+TM(\log T)^2}
&\geqsim&
\min\bigg\{
T{\cred\wt{\gamma}_n}^{\frac{{\fred\sfd}}{{\colorro2}\beta}}(\log n)^{-1},\>
\big(T{\cred\wt{\gamma}_n}^{\frac{\sfd}{{\colorro2}\beta}}\big)^2,\>
T{\cred\wt{\gamma}_n}^{\frac{\sfd}{{\colorro2}\beta}}(\log T)^{-2}
\bigg\}. 
\eeas
%
By {\cred(\ref{202604131857a})}, in particular, 
$
T{\cred\wt{\gamma}_n}^{\frac{\sfd}{{\colorro2}\beta}} \to \infty
$, 
and hence we have
\bea\label{202604131858}
\frac{T^2}{v^2T+M^2+TM(\log T)^2}
&\geqsim&
\min\bigg\{
T{\cred\wt{\gamma}_n}^{\frac{{\fred\sfd}}{{\colorro2}\beta}}(\log n)^{-1},\>
T{\cred\wt{\gamma}_n}^{\frac{\sfd}{{\colorro2}\beta}}(\log T)^{-2}
\bigg\}
\nn\\&\simgeq&
T{\cred\wt{\gamma}_n}^{\frac{{\fred\sfd}}{{\colorro2}\beta}}(\log n)^{-2}
\eea
{\cred since $n\geq T=nh$ for large $n$ due to $h\to0$.} %
Condition {\cred(\ref{202604131857a})} is requesting that the right-hand side of the (\ref{202604131858}) is growing faster than $\log n$, i.e., 
\beas
T{\cred\wt{\gamma}_n}^{\frac{{\fred\sfd}}{{\colorro2}\beta}}(\log n)^{-2} &\gg& \log n
\eeas
as $n\to\infty$. 

The probability ${\colorro Q_{1,n}(a) }$ of (\ref{202607011605}) is of order 
$
h^{-1}\exp\big(-\Csbk T{\cred\wt{\gamma}_n}^{\frac{{\fred\sfd}}{\beta}}(\log n)^{-2}\big)
$
for sufficiently large $n$, for some positive constant $\Csbk$, 
which follows from the large deviation inequality for $\alpha$-mixing processes 
of Theorem 2 of Merlev\`ede et al. \cite{merlevede2009bernstein}. 
{\cred 
Note that the estimate is uniform in $a\in\calx$. 
Therefore, Condition (\ref{202604131857a}) ensures that 
\bea\label{202608071441}
d_n^{-\sfd}\sup_{a\in{\cred\check{\calx}}}Q_{1,n}(a)=O(n^{-L})
\eea 
as $n\to\infty$, for any $L>0$, 
{\cred because 
$h^{-1}=n/T\ll n$ by $T\to\infty$, and 
\bea\label{202608071435}
d_n^{-\sfd}=\wt{\gamma}_n^{-\frac{\sfd}{2\beta}}\leq T(\log n)^{-3}\leq n  
\eea
for large $n$, by (\ref{202604131857a}).
(\ref{202608071435}) also gives $d_n^{2\beta}\geq n^{-2\beta/\sfd}$ for large $n$. 
}
}
{\cred
From (\ref{202607011407}) and (\ref{202608071441}), we obtain the estimate }
\bea\label{202607011624}
\bbR_n\big((\ol{X}_\tjm)_{j\in\bbI_n}\big) 
&\leq& 
\Csbj
\big[
\bbR_n\big((X_\tjm)_{j\in\bbI_n}\big) 
+d_n^{2\beta}
\big]
\quad(n\in\bbN)
\eea
for some constant $\Csbj$. 
Since 
$\bbR_n\big((\ol{X}_\tjm)_{j\in\bbI_n}\big)$ is compatible with $\ol{\bfR}_n^p$, and 
$\bbR_n\big((X_\tjm)_{j\in\bbI_n}\big)$ with $\bfR_n^p$, 
due to (\ref{202505050127}), we obtain Inequality (\ref{202509240317}) {\cred of (a)} from (\ref{202607011624}). 

{\cred 
To show (b), we set 
\beas 
\wt{\gamma}_n=T^{-2\beta/\sfd}(\log n)^{6\beta/\sfd}
\bigg[\call_n\wedge\big\{T(\log n)^{-3}\big\}^{\frac{\beta}{\sfd}}\bigg].
\eeas 
Then $\lim_{n\to\infty}\wt{\gamma}_n=0$ and 
the condition (\ref{202604131857a}) is satisfied. 
Therefore, (b) follows from (a). 

Set $\wt{\gamma}_n=\gamma_n$. Then (\ref{202604131857a}) is satisfied with (\ref{202604131857}), 
and hence (c) is proved by (a). 
}
\qed
\halflineskip
\halflineskip

\subsection{Estimation of $\ol{\bfR}_n$ by $\ol{\bfR}_n^p$}
We give a result similar to Lemma \ref{202602140214}. 
\begin{lemma}\label{202605140839}
There exists a constant $\Csat$ such that 
\bea\label{202602140207}
\ol{\bfR}_n
&\leq&
\Csat
{\cred
\big(
\ol{\bfR}_n^p
+n^{-1}+h_n^2\big).
}
\eea
\end{lemma}
\proof
In the same way as the proof of Lemma \ref{202602140214}, 
we consider the decomposition 
\bea\label{202602140054g}
\ol{\bfR}_n-\ol{\bfR}_n^p
&=& 
E[\ol{\bbJ}_n]+{\colorro E[\ol{r}_n]}
{\cred\>+O(n^{-1})},
\eea
which is the generalized version of (\ref{202602140054}). 
Here we used estimates similar to those provided by Lemma \ref{202602132307}, 
but $\ol{\bbJ}_n$ is defined as 
\bea\label{202602120111g}
\ol{\bbJ}_n
\yeq
\frac{1}{2n} \sum_{j\in\bbI_n}\big(\wh{S}_n(\ol{X}_\tjm)^{-1}-S^*(\ol{X}_\tjm)^{-1}\big)
 \big[\ol{M}^n_j(\tj{\colorro\wedge\ol{\varpi}_j})\big]
\onecalw(\ol{X}_\tjm), 
\eea
which is the generalized version of $\bbJ_n$ of (\ref{202602120111}), 
with 
\bea\label{202602130220g}
\ol{M}^n_j(u)
&=&
h^{-1}\bigg\{\bigg(\int_\tjm^u\sigma^*(\ol{X}_t)d\ol{w}_t\bigg)^{\otimes2}-\int_\tjm^u\sigma^*(\ol{X}_t)^{\circledast2}dt\bigg\}
\nn\\&&
+2b^*(\ol{X}_\tjm)\wt{\otimes}\int_\tjm^u\sigma^*(\ol{X}_t)d\ol{w}_t
+h^{-1}\int_\tjm^u(\tj-t)\partial S^*(\ol{X}_t){\colorro [{\cred\sigma}_s^{\ol{X}}d\ol{w}_s^{\ol{X}}]},
\eea
{\cred 
where $\ol{w}$ and $\ol{w}^{\ol{X}}$ are Wiener processes driving $(\ol{X},\ol{Y})$, and $\sigma^{\ol{X}}$ and $\ol{\varpi}_j$ 
are corresponding to $\sigma^X$ and ${\varpi}_j$, respectively. 
}
{\colorro 
The random variable $\ol{r}_n$ is also defined in a similar fashion as 
\beas
\ol{r}_n
&=&
\frac{1}{2n} \sum_{j\in\bbI_n}\big({\cred\wh{S}_n}(\ol{X}_\tjm)^{-1}-S^*(\ol{X}_\tjm)^{-1}\big)
 \big[\ol{\dot{r}}^n_j1_{\{\ol{\varpi}_j=\tj\}}\big]
\onecalw(\ol{X}_\tjm).
\eeas
with the counterpart $\ol{\dot{r}}^n_j$ of $\dot{r}^n_j$. 
}
Note that the estimates provided by Lemma \ref{202602132307} are unchanged for $(\ol{w},{\colorro\ol{w}^X,}\ol{X},\ol{Y})$ 
since this {\cred quartet} is an independent copy of $(w,{\colorro w^X,}X,Y)$.

{\cred 
We have $E[\ol{\bbJ}_n]=0$ by the independency of the random variables. 
On the other hand, 
by the same estimate of $\ol{\dot{r}}^n_j1_{\{\ol{\varpi}_j=\tj\}}$ as (\ref{202606301728}) for $\dot{r}^n_j1_{\{\ol{\varpi}_j=\tj\}}$, we obtain 
\beas 
E\big[|\ol{r}_n|\big]
&\simleq&
\big(E\big[|\wh{S}_n-S^*|_n^2\big]\big)^{1/2}\bigg(E\bigg[n^{-1}\sum_{j\in\bbI_n}\big|\ol{\dot{r}}^n_j1_{\{\ol{\varpi}_j=\tj\}}\big|^2\bigg]\bigg)^{1/2}
\nn\\&\simleq&
\big(\ol{\bfR}_n^p\big)^{1/2}h
\nn\\&\leq&
\ol{\bfR}_n^p+h^2. 
\eeas
Consequently, it follows from (\ref{202602140054g}) that 
\beas
\ol{\bfR}_n-\ol{\bfR}_n^p
&\simleq& 
\ol{\bfR}_n^p+h^2+O(n^{-1})
\eeas
This proves the inequality (\ref{202602140207}). 
}
%
\qed
\halflineskip

\subsection{Estimation of $E\big[\Phi_n(S,Z)-\Phi_n(S^*,Z)\big]$}

{\cred 

Recall that $\ol{U}_n(S)$ is defined in (\ref{202608081238}). 

\begin{lemma}\label{202607131330}
{\fred There} exists a constant $\Csbb$ such that 
\bea\label{202607131322}
E\big[\Phi_n(S,Z)-\Phi_n(S^*,Z)\big]
&\leq&
\Csbb\big(
E\big[\ol{U}_n(S)\big]+h^2+n^{-1}\big)
\eea
for all $S\in\bbS$ and ${\fred n}\in\bbN$. 
\end{lemma}
\proof
Let $S\in\bbS$. Then 
\bea\label{202607131319a} &&
E\big[\Phi_n(S,Z)-\Phi_n(S^*,Z)\big]-E\big[\ol{U}_n(S)\big]
\nn\\&=&
E\bigg[\frac{1}{2n}\sum_{j\in\bbI_n}
\big(S(X_\tjm)^{-1}-S^*(X_\tjm)^{-1}\big)
[h^{-1}(\Delta_jY)^{\otimes2}-S^*(X_\tjm)]\calw(X_\tjm)\bigg]
\nn\\&=&
\bbK_{1,n}(S)+\bbK_{2,n}(S),
\eea
where 
\beas 
\bbK_{1,n}(S)
&=& 
E\bigg[\frac{1}{2n}\sum_{j\in\bbI_n}
\big(S(X_\tjm)^{-1}-S^*(X_\tjm)^{-1}\big)
[M^n_j(\tj\wedge\tau_j)+\dot{r}^n_j]\calw(X_\tjm)1_{\{\tau_j=\tj\}}\bigg]
\eeas
and 
\beas 
\bbK_{2,n}(S)
&=& 
E\bigg[\frac{1}{2n}\sum_{j\in\bbI_n}
\big(S(X_\tjm)^{-1}-S^*(X_\tjm)^{-1}\big)
[h^{-1}(\Delta_jY)^{\otimes2}-S^*(X_\tjm)]\calw(X_\tjm)1_{\{\tau_j<\tj\}}\bigg]. 
\eeas
We used Lemma \ref{202602132307} (a) for $\bbK_{1,n}$. 
By (\ref{202607031437}) and $[A1]$ (iv), we have  
\bea\label{202607131319b}
\sup_{S\in\bbS}\big|\bbK_{2,n}(S)\big| &=& O(n^{-L})
\eea
as $n\to\infty$, for any $L>0$. 
Similarly, we have 
\beas 
\bbK_{1,n}(S)
&=& 
E\bigg[\frac{1}{2n}\sum_{j\in\bbI_n}
\big(S(X_\tjm)^{-1}-S^*(X_\tjm)^{-1}\big)
[M^n_j(\tj\wedge\tau_j)+\dot{r}^n_j1_{\{\tau_j=\tj\}}]\calw(X_\tjm)\bigg]
+\ol{O}(n^{-L})
\nn\\&=& 
E\bigg[\frac{1}{2n}\sum_{j\in\bbI_n}
\big(S(X_\tjm)^{-1}-S^*(X_\tjm)^{-1}\big)
[\dot{r}^n_j1_{\{\tau_j=\tj\}}]\calw(X_\tjm)\bigg]
+\ol{O}(n^{-L}), 
\eeas
where $\ol{O}$ indicates the order is uniform in $S\in\bbS$. 
Thus, by (\ref{202606301728}), there exists onstants $\Csba$ and $\Csbl$ such that 
\bea\label{202607131319c}
\big|\bbK_{1,n}(S)\big|
&\leq&
2\Csba \big(E\big[|S-S^*|_n^2\big]\big)^{1/2}h+\ol{O}(n^{-L})
\nn\\&\leq&
\Csba \big(E\big[|S-S^*|_n^2\big]+h^2\big)+\ol{O}(n^{-L})
\nn\\&\leq&
\Csbl \big(E\big[\ol{U}_n(S)\big]+h^2\big)+\ol{O}(n^{-L})
\eea
for all $S\in\bbS$ and $n\in\bbN$. 
From (\ref{202607131319a}), (\ref{202607131319b}), (\ref{202607131319c}) and the compatibility, 
we obtain the inequality (\ref{202607131322}). 
\qed
\halflineskip
}

\subsection{Proof of {\cred Theorems \ref{202602140323} 
and Propositions \ref{202607131357}}}\label{202606261834}

{\it Proof of Theorem \ref{202602140323}.}
{\cred Combine Lemmas \ref{202602140253} and \ref{202607131330} as well as Lemma \ref{202509220400} (c) 
to obtain the inequality (\ref{202602140245a}) of Theorem \ref{202602140323} for $\ol{\bfR}_n^p$. 
Then, the inequality (\ref{202602140245a}) for $\ol{\bfR}_n$
is obtained from the inequality for $\ol{\bfR}_n^p$
with the aid of Lemma  \ref{202605140839}. 
}
\qed\halflineskip

{\cred 
\noindent
{\it Proof of Proposition \ref{202607131357}.}
Combine Lemmas \ref{202602140253} and \ref{202607131330} as well as Lemma Lemma \ref{202509220400} (b) 
to show (\ref{202602140245a_flat}) of Proposition \ref{202607131357} for $\ol{\bfR}_n^p$. 
Then we obtain (\ref{202602140245a_flat}) for $\ol{\bfR}_n$
from that for $\ol{\bfR}_n^p$
with the aid of Lemma  \ref{202605140839}.  
\qed\halflineskip

}

\section{Application to Deep Learning}\label{202607151446}
\subsection{An upper bound of the generalization error}\label{202607151446a}

{\cred 
In this section, we use deep learning to obtain the initial estimator $\wh{S}^\DL_n$. 
For ${\bf v}=(v_1,...,v_d)^\star$, 
we define the shifted ReLU activation function $\sigma_v:\bbR^d\to\bbR^d$ as 
\beas 
\sigma_v(x) &=& \big((x_1-v_1)_+,...,(x_d-v_d)_+\big)^\star
\qquad(x=(x_1,...,x_d)^\star\in\bbR^d)
\eeas
where $u_+=\max\{u,0\}$ for $u\in\bbR$. 
Define the map $f\big(\cdot;(W_L,...,W_1,W_0),({\bf v}_L,...,{\bf v}_1)\big):\bbR^{{\fred\sfp(0)}}\to\bbR^{{\fred\sfp(L+1)}}$ by 
\bea\label{202412251346}
f\big(x;(W_L,...,W_1,W_0),({\bf v}_L,...,{\bf v}_1)\big)
&=&
W_L\sigma_{{\bf v}_L}W_{L-1}\sigma_{{\bf v}_{L-1}}\cdots W_1\sigma_{{\bf v}_1}W_0x
\quad(x\in\bbR^{\sfp_0}). 
\eea
for weight matrices $W_i\in\bbR^{{\fred\sfp(i+1)}}\otimes\bbR^{{\fred\sfp(i)}}$ ($i=0,1,...,L$) and 
shift vectors ${\bf v}_i\in\bbR^{{\fred\sfp(i)}}$ ($i=1,...,L$). 
%

Let $\calx=[0,1]^\sfd$. 
We denote by $\cald$ the set of functions $f$ that of the form (\ref{202412251346}) 
for some $L$, $(W_L,...,W_1,W_0)$ and $({\bf v}_L,...,{\bf v}_1)$. 
We consider the leaning machine ${\mathfrak F}_n$ specified by 
\beas 
{\mathfrak F}_n
&=& 
\bigg\{f\in\cald\text{ of the form }(\ref{202412251346});\>
\max_{\ell=0,...,L_n, j=1,...,L_n}\big(\|W_\ell\|_\infty\vee\|{\bf v}_j\|_\infty\big)\leq1,
\nn\\&&
\sum_{\ell=0}^{L_n}\|W_\ell\|_0+\sum_{j=1}^{L_n}\|{\bf v}_j\|_0\leq s_n,\>\|f\|_\infty\leq \check{K}
\bigg\}
\eeas
for some $\check{K}>0$ and 
\bea\label{202607141144}
s_n,\  L_n, 
{\fred \sfp_n(0)=\sfd}, 
\sfp_{n}(1),....\sfp_n(L_n) \in \bbN, 
{\fred \sfp_n(L_n+1)=\sfd^2}.
\eea
{\fred Here $\sfd$ is the dimension of the covariate process $X$, and 
the efficient dimension is $\sfd(\sfd+1)/2$ for $\sfp_n(L_n+1)$ by symmetry. 
We write $\sfp_n(i)$ for $\sfp(i)$ to specify the dependency of $\sfp(i)$ on $n$.} 

The $\infty$-norm $\|\cdot\|_\infty$ denotes the maximum-entry norm of the object and %
the $0$-norm $\|\cdot\|_0$ denotes the number of non-zero entries of the object. 
The class ${\mathfrak F}_n$ is used by Schmidt-Hieber \cite{schmidt2020nonparametric} 
as a deep neural network with ReLU activation function 
by appropriately controlling the numbers in (\ref{202607141144}).  

For a rectangle $D$ in an Euclidean space and a number $K$, 
let 
\beas
B^\beta(D,K)
&=& 
\big\{f\in C(D);\>
\|f\|_{C_b^{\ell,\beta-\ell}}(\text{Int}(D))\leq K\big\}. 
\eeas
As Schmidt-Hieber \cite{schmidt2020nonparametric}, we assume that each component of the true function $S^*$ admits 
a representation as 
\beas 
S^* &=& {\tt S}_q\circ\cdots\circ{\tt S}_0
\eeas
with some ${\tt S}_i=({\tt S}_{i,j})_{j\in[d_{i+1}]}:[a_i,b_i]^{d_i}\to[a_{i+1},b_{i+1}]^{d_{i+1}}$ 
satisfying ${\tt S}_{i,j}\in B^{\beta_i}([a_i,b_i]^{\sfq_i},K)$ 
for all $j\in[d_{i+1}]$, with some $\sfq_i\in\bbN$, $a_i$ and $b_i$ such that  $|a_i|, |b_i|\leq K$. 
Here $d_0=\sfd$ and $d_q=\sfd^2$. 

{\cblue We assume that all $\beta_i\geq2$. 
Then the effective smoothness index in Schmidt-Hieber \cite{schmidt2020nonparametric} 
is 
$\beta^*_i\equiv\beta_i\prod_{j=i+1}^q(\beta_j\wedge1)=\beta_i$, and the key convergence rate becomes 
\beas 
\phi_n &=& \max_{i=0,...,q}n^{-\frac{2\beta_i}{2\beta_i+\sfq_i}}.
\eeas
Let $\beta=\min_{i=0,...,q}\beta_i$. 
}

To reduce the generalization error, we need to control the size of ${\mathfrak F}_n$ 
so as it becomes rich but not too large. 
As Schmidt-Hieber \cite{schmidt2020nonparametric}, we impose the following conditions:
\bea\label{202607141413}&&
\check{K}\geq\max\{K,1\}, \quad
\sum_{i=0}^q\log_2 4(\sfq_i\vee\beta_i)\log_2n\leq L_n\simleq n\phi_n, \quad
\nn\\&&
n\phi_n\simleq\min\{\sfp_n(1),...,\sfp_n(L_n)\}, \quad
s_n\asymp n\phi_n\log n. 
\eea
as $n\to\infty$. 

By Inequality (26) of Schmidt-Hieber \cite{schmidt2020nonparametric} and the compatibility, we obtain 
\bea\label{202607141319}
\inf_{S\in\mathfrak{F}_n}{\cred E\big[\ol{U}_n(S) \big]}
&\simleq&
\phi_n. 
\eea
On the other hand, we have 
an estimate of the covering number as 
\bea\label{202412251659}
\log\caln_n
&\leq&
{\cblue \sfd^2}
(s_n+1)\log\bigg[2\delta_n^{-1}(L_n+1)\prod_{\ell=0}^{L_n+1}(\sfp_n(\ell)+1)\bigg]
\eea
due to Lemma 5 of Schmidt-Hieber \cite{schmidt2020nonparametric}. 

Similarly to Theorem 1 of Schmidt-Hieber \cite{schmidt2020nonparametric},  we obtain 
a bound of the generalization error. 

\begin{theorem}\label{202607141406}
Suppose that Conditions $[A1]$-$[A3]$ and 
the conditions in (\ref{202607141413}) are satisfied. 
Moreover, suppose that there exists a constant $\Csbe$ such that 
\bea\label{202607141509}
\Delta_n+\>\wh{\Delta}_n\leq \Csbe\big(\phi_nL_n(\log n)^2+h^2\big)
\eea
for $n\in\bbN$. 
Then, there exists a constant $\Csbf$ such that 
\bea\label{202607141459}
\ol{\bfR}_n \vee \ol{\bfR}_n^p 
&\leq&
\Csbf \big(\phi_nL_n(\log n)^2+h^2\big)
\eea
for $n\in\bbN$. 
\end{theorem}
\proof 
Under the sparsity with $s_n$, 
any initial estimator $\wh{S}_n^0$ can find a counterpart estimator in the architecture satisfying 
$\sfp_n(1),....,\sfp_n(L_n)\leq s_n$. 
So we may assume that $\sfp_n(1),....,\sfp_n(L_n)\simleq n$. 
In view of (\ref{202607141319}) and (\ref{202412251659}) for $\delta_n=1/n$, 
we obtain (\ref{202607141459}) from 
{\fred Theorem \ref{202602140323}}.
\qed
\halflineskip

According to the second condition of (\ref{202607141413}), 
we can choose an architecture satisfying 
\beas 
\sum_{i=0}^q\log_2 4(\sfq_i\vee\beta_i)\log_2n\approx L_n.
\eeas
Then, for a constant $\Csbg$, we obtain 
\bea\label{202607141504}
\ol{\bfR}_n \vee \ol{\bfR}_n^p 
&\leq&
\Csbg \big(\phi_n(\log n)^3+h^2\big)
\eea
for $n\in\bbN$, 
under Conditions $[A1]$-$[A3]$ and (\ref{202607141509}). 
}

\subsection{Lower bound
}\label{202607151446b}
A lower bound will be {\fred presented} for the model (\ref{202509191103}) with a {\cred stationary} covariate process $X$ independent of $w$, 
that is, $Y$ is a doubly stochastic process. 
In this situation, we may assume that $X$ is $\calf_0$-measurable. 
Given $\calf_0$, the random variables $\{\Delta_jY\}_{j\in\bbI_n}$ are independent and 
\beas
\Delta_jY&\sim& N_\sfm\bigg(\int_\tjm^\tj b(X_t)dt,\>\int_\tjm^\tj S(X_t)dt\bigg).
\eeas
Therefore, the $\calf_0$-conditional probability density of $(\Delta_jY)_{j\in\bbI_n}$ is 
\beas 
p_n^S(z_1,...,z_n)
&=& 
\prod_{j\in\bbI_n}\phi\bigg(z_j;\int_\tjm^\tj b(X_t)dt,\>\int_\tjm^\tj S(X_t)dt\bigg). 
\eeas
$E^{P^S}$ denotes {\fred the expectation with respect to }%
the probability distribution $P^X(dx)P^{(\Delta_jY)_{j\in\bbI_n}}(\cdot|x)$ for $Y$ corresponding to $S$. 
The Kullback-Leibler divergence is 
\beas 
\text{KL}(P^S,P^{S_0})
&=&
E^{P^S}\bigg[\log\frac{p_n^S({\fred\Delta_1Y,...,\Delta_nY})}{p_n^{S_0}({\fred\Delta_1Y,...,\Delta_nY})}\bigg]
\nn\\&=&
-\half\sum_{j\in\bbI_n}E^{P^S}\bigg[\hspace{-5pt}\bigg[\log\frac{\det\int_\tjm^\tj S(X_t)dt}{\det\int_\tjm^\tj S_0(X_t)dt}
\nn\\&&
+\bigg(\int_\tjm^\tj S(X_t)dt\bigg)^{-1}\bigg[\bigg({\fred\Delta_jY}-\int_\tjm^\tj b(X_t)dt\bigg)^{\otimes2}\bigg]
\nn\\&&
-\bigg(\int_\tjm^\tj S_0(X_t)dt\bigg)^{-1}\bigg[\bigg({\fred\Delta_jY}-\int_\tjm^\tj b(X_t)dt\bigg)^{\otimes2}\bigg]\bigg]\hspace{-5pt}\bigg]
\nn\\&=&
-\half\sum_{j\in\bbI_n}E\bigg[\hspace{-5pt}\bigg[\log\frac{\det\int_\tjm^\tj S(X_t)dt}{\det\int_\tjm^\tj S_0(X_t)dt}
\nn\\&&
+\bigg\{\bigg(\int_\tjm^\tj S(X_t)dt\bigg)^{-1}-\bigg(\int_\tjm^\tj S_0(X_t)dt\bigg)^{-1}\bigg\}
\bigg[\int_\tjm^\tj S(X_t)dt\bigg]\bigg]\hspace{-5pt}\bigg].
\eeas
We assume that the support $\text{supp}(X_0)$ is 
in $\calw=[0,1]^\sfd$ and that 
the law of $X_0$ has density that is  bounded from above and below. 
Then, there exists a positive constant $\Csi$ such that 
\beas 
\text{KL}(P^S,P^{S_0})
&=&
\half nE\bigg[\hspace{-5pt}\bigg[
\bigg\{\bigg(\int_0^h S_0(X_t)dt\bigg)^{-1}-\bigg(\int_0^h S(X_t)dt\bigg)^{-1}\bigg\}
\bigg[\int_0^h S(X_t)dt\bigg]
\nn\\&&\hspace{50pt}
-\log\frac{\det\int_0^h S(X_t)dt}{\det\int_0^h S_0(X_t)dt}
\bigg]\hspace{-5pt}\bigg]
\nn\\&\leq&
\Csi nE\bigg[
\bigg|h^{-1}\int_0^h S_0(X_t)dt-h^{-1}\int_0^h S(X_t)dt\bigg|^2\bigg]
\nn\\&\leq&
\Csi nE\big[|S_0(X_0)-S(X_0)|^2\big]
\nn\\&=&
\Csi nE\big[|S_0(X_0)-S(X_0)|^2\onecalw(X_0)\big]
\eeas
by Jensen's inequality. 
Now the situation is the same as the proof of Theorem 3 of {\fred Schmidt-Hieber \cite{schmidt2020nonparametric}}. 
Thus, we obtain the mini-max bound. Namely, in his terminology, 
\beas 
\inf_{\wt{S}_n}\sup_{S\in\calg(q,{\bf d},{\bf t},{\bm \beta},K)}
{\fred\>\ol{\bfR}_n^p({\fred\wt{S}}_n,S)}
&\geq&
\csj\phi_n,     
\eeas
where $\wt{S}_n$ denotes any estimator of $S$ and 
$\csj$ is some positive constant. 
{\fred 
Regarding the notation, $\calg(q,{\bf d},{\bf t},{\bm \beta},K)$ is specifically the class adopted by Schmidt-Hieber \cite{schmidt2020nonparametric} to prove the minimax bound.}

\appendix
\section{Compatibility}\label{202604101629}
The condition (\ref{202505050251}) implies 
\bea\label{202504260010}
0 &<& \inf_{{\colorro x\in\calx, S\in\bbS}}\lambda_{\min}\big(S^*(x)^{1/2}S(x)^{-1}S^*(x)^{1/2}\big)
\nn\\
&\leq&
\sup_{{\colorro x\in\calx, S\in\bbS}}\lambda_{\max} \big(S^*(x)^{1/2}S(x)^{-1}S^*(x)^{1/2}\big)<\infty. 
\eea

Denote by $\lambda_i$ the $i$-th eigenvalue of the symmetric matrix ${\tt M}:=S^*(x)^{1/2}S(x)^{-1}S^*(x)^{1/2}$ 
for $i=1,...,\sfm$.  
By the condition (\ref{202504260010}), there exist constants $\lambda_-\in(0,1)$ and $\lambda_+\in(1,\infty)$ such that 
\beas 
0\><\>\lambda_-\yleq\min_{i=1,...,\sfm}\lambda_i\yleq\max_{i=1,...,\sfm}\lambda_i\yleq \lambda_+
\eeas
for all $S\in\bbS$ and $x\in\calw$. 
For the constants $\lambda_-$ and $\lambda_+$, 
there exist positive numbers $c_-$ and $c_+$ such that 
\beas 
c_-(y-1)^2\yleq {\colorro\half\big(}y-1-\log y{\colorro\big)}\yleq c_+(y-1)^2\quad(y\in[\lambda_-,\lambda_+]). 
\eeas

Since 
\beas 
\lambda_+^{-2}\sum_{i=1}^\sfm(\lambda_i-1)^2
\yleq
\sum_{i=1}^\sfm(\lambda_i^{-1}-1)^2
\yleq
\lambda_-^{-2}\sum_{i=1}^\sfm(\lambda_i-1)^2,
\eeas
$\sum_{i=1}^\sfm(\lambda_i-1)^2=\text{Tr}\big\{\big({\tt M}-I_\sfm)^2\big\}$ 
and $\sum_{i=1}^\sfm(\lambda_i^{-1}-1)^2=\text{Tr}\big\{\big({\tt M}^{-1}-I_\sfm)^2\big\}=|{\tt M}^{-1}-I_\sfm|^2$ 
for the $\sfm$-dimensional identity matrix $I_\sfm$, 
and 
\beas 
U (x,S)&=&
\half\bigg\{
\big(S(x)^{-1}-S^*(x)^{-1}\big)[S^*(x)]+\log\frac{\det S(x)}{\det S^*(x)}\bigg\}\onecalw(x)
\nn\\&=& 
\half\bigg\{
\text{Tr}\big(S(x)^{-1}S^*(x)\big)-\sfm-\log\det\big(S(x)^{-1}S^*(x)\big)\bigg\}\onecalw(x)
\nn\\&=& 
\half
\sum_{i=1}^\sfm\big(\lambda_i-1-\log\lambda_i\big)\onecalw(x), 
\eeas
we have the compatibility on $\calx$, that is, 
there exists a positive constant $C_*\geq1$ such that 
{\colorro
\bea\label{202505050127} 
C_*^{-1}|S(x)-S^*(x)|^2\onecalw(x)
\yleq
U (x,S)
\yleq
C_*|S(x)-S^*(x)|^2\onecalw(x)
\eea
}
for all $x\in\calx$ and $S\in\bbS$. 

In fact, 
\beas 
U (x,S)
&\leq&
c_+\lambda_+^2|{\tt M}^{-1}-I_\sfm|^2{\colorro\onecalw(x)}
\nn\\&=&
c_+\lambda_+^2|S^*(x)^{-1/2}(S(x)-S^*(x))S^*(x)^{-1/2}|^2{\colorro\onecalw(x)}
\nn\\&\leq&
c_+\lambda_+^2|S^*(x)^{-1/2}|^{\colorro4}|S(x)-S^*(x)|^2{\colorro\onecalw(x)}
\nn\\&\leq&
c_+\lambda_+^2{\colorro\bigg(}\frac{\sfm}{\inf_{x\in{\colorro\calx}}\lambda_{\min}(S^*(x))}{\colorro\bigg)^2}|S(x)-S^*(x)|^2{\colorro\onecalw(x)}
\eeas
due to 
\beas 
|S^*(x)^{-1/2}|^2 
&=& 
\text{tr}\big(S^*(x)^{-1}\big)
\yleq 
\frac{\sfm}{\inf_{x\in{\colorro\calx}}\lambda_{\min}(S^*(x))}, 
\eeas
and also 
\bea\label{202604061409}
U (x,S)
&\geq&
c_-\lambda_-^2|{\tt M}^{-1}-I_\sfm|^2{\colorro\onecalw(x)}
\nn\\&=&
c_-\lambda_-^2|S^*(x)^{-1/2}(S(x)-S^*(x))S^*(x)^{-1/2}|^2{\colorro\onecalw(x)}
\nn\\&\geq&
c_-\lambda_-^2\bigg({\colorro\sfm}\sup_{x\in{\colorro\calx}}\lambda_{\max}S^*(x)\bigg)^{-2}\>\big|S(x)-S^*(x)\big|^2{\colorro\onecalw(x)}.
\eea

\section{Extension of functions}
\begin{lemma}\label{202604121758}
Suppose the following conditions. 
\bd
\im[(i)] $D$ is an open set of $\bbR^\sfd$. 
\im[(ii)]  $W\in C_b^{\ell,\beta-\ell}(\bbR^\sfd)$, $\text{supp }W\subset\ol{D}$. 
\im[(iii)] $f\in C_b^{\ell,\beta-\ell}(D)$.
\ed
Define the function $g:\bbR^\sfd\to\bbR$ as 
\beas 
g(x) 
&=& 
\l\{\begin{array}{cc}
f(x)W(x)&(x\in D)\y
0&(x\in D^c).
\end{array}\r.
\eeas
Then $g\in C_b^{\ell,\beta-\ell}(\bbR^\sfd)$. 
\end{lemma}

\bibliographystyle{spmpsci}      
\bibliography{bibtex-20260714}   

\end{document}

%% file: nakamacro300823-300916+.tex
\def\koko{{\coloroy{koko}}}
\def\bd{\begin{description}}
\def\ed{\end{description}}

\def\D2{\bbD_{2,\infty-}}

\def\tj{{t_j}}
\def\tjm{{t_{j-1}}}

\def\bb{\bar{B}}

\def\A{{\bf A}}

\def\D{{\bf D}}

\def\M{{\bf M}}

\def\V{{\bf V}}

\def\cala{{\cal A}}
\def\calb{{\cal B}}
\def\calc{{\cal C}}
\def\cald{{\cal D}}
\def\cale{{\cal E}}
\def\calf{{\cal F}}
\def\calg{{\cal G}}

\def\call{{\cal L}}
\def\calm{{\cal M}}
\def\caln{{\cal N}}
\def\calo{{\cal O}}

\def\calr{{\cal R}}
\def\cals{{\cal S}}

\def\calu{{\cal U}}

\def\calw{{\cal W}}
\def\calx{{\cal X}}

\def\yeq{\>=\>}
\def\yleq{\>\leq\>}
\def\ygeq{\>\geq\>}

\def\sfk{{\sf k}}
\def\sfm{{\sf m}}

\def\sfd{{\sf d}}
\def\sfp{{\sf p}}
\def\sfr{{\sf r}}

\def\simleq{\ \raisebox{-.7ex}{$\stackrel{{\textstyle <}}{\sim}$}\ }
\def\geqsim{\ \raisebox{-.7ex}{$\stackrel{{\textstyle <}}{\sim}$}\ }
\def\simgeq{\ \raisebox{-.7ex}{$\stackrel{{\textstyle >}}{\sim}$}\ }
\def\geqsim{\ \raisebox{-.7ex}{$\stackrel{{\textstyle >}}{\sim}$}\ }
\def\ep{\epsilon}
\def\half{\frac{1}{2}}

\def\up{\uparrow}
\def\down{\downarrow}

\def\y{\vspace*{3mm}\\}
\def\halflineskip{\vspace*{3mm}}
\def\nn{\nonumber}
\def\be{\begin{equation}}
\def\ee{\end{equation}}
\def\bea{\begin{eqnarray}}
\def\eea{\end{eqnarray}}
\def\beas{\begin{eqnarray*}}
\def\eeas{\end{eqnarray*}}
\def\bi{\begin{itemize}}
\def\ei{\end{itemize}}
\def\im{\item}
\def\bd{\begin{description}}
\def\ed{\end{description}}
\def\l{\left}
\def\r{\right}

\def\dots{\stackrel{\circ}{S}}
\def\dotc{\stackrel{\circ}{C}}

\newcommand{\bbA}{{\mathbb A}}
\newcommand{\bbB}{{\mathbb B}}
\newcommand{\bbC}{{\mathbb C}}
\newcommand{\bbD}{{\mathbb D}}

\newcommand{\bbF}{{\mathbb F}}

\newcommand{\bbI}{{\mathbb I}}
\newcommand{\bbJ}{{\mathbb J}}
\newcommand{\bbK}{{\mathbb K}}
\newcommand{\bbL}{{\mathbb L}}
\newcommand{\bbM}{{\mathbb M}}
\newcommand{\bbN}{{\mathbb N}}

\newcommand{\bbR}{{\mathbb R}}
\newcommand{\bbS}{{\mathbb S}}
\newcommand{\bbT}{{\mathbb T}}

\newcommand{\bbZ}{{\mathbb Z}}